\documentclass[12pt,a4paper]{article}
\usepackage[english]{babel}
\usepackage[T1]{fontenc}
\usepackage{graphicx}
\usepackage{epstopdf}
\usepackage{epsfig,psfrag}
\usepackage{amsmath, amsfonts, amsthm,amssymb}
\usepackage{hyperref}
\newtheorem{definition}{Definition}[section]
\newtheorem{lemma}{Lemma}[section]

\newtheorem{theorem}{Theorem}[section]
\newtheorem{corollary}{Corollary}[section]
\newtheorem{remark}{Remark}[section]

\newtheorem{proposition}{Proposition}[section]

\numberwithin{equation}{section}
\def\Z{\mathbb Z}
\def\R{\mathbb R}
\def\Q{\mathbb Q}
\def\T{\mathbb T}
\def\RR{\mathcal R}
\def\SS{\mathcal S}
\def\FF{\mathcal F}
\def\DD{\mathcal D}
\def\XX{\mathcal X}
\def\OO{\mathcal O}

\title{Growth of Sobolev norms for the quintic NLS on $\T^2$}
\author{E. Haus*, M. Procesi*
\vspace{2mm} 
\\ \small 
$^*$ Dipartimento di Matematica, Universit\`a di Roma ``La Sapienza", Roma, I-00185, Italy
\\ \small 
E-mail:  \texttt{haus@mat.uniroma1.it}, \texttt{mprocesi@mat.uniroma1.it}}

\begin{document}
\maketitle
\abstract{We study the quintic Non Linear Schr\"odinger equation on a two dimensional torus and exhibit orbits whose Sobolev norms grow with time. The main point is to reduce to a sufficiently simple toy model, similar in many ways to the one discussed in \cite{CKSTT} for the case of the cubic NLS. This requires an accurate combinatorial analysis. }

\section{Introduction}

We consider the quintic defocusing NLS on the two-dimensional torus $\mathbb{T}^2=\mathbb R^2/(2\pi\mathbb Z)^2$
\begin{equation}\label{nls}
-i\partial_tu+\Delta u=|u|^4u\ ,
\end{equation}
which is an infinite dimensional dynamical system with Hamiltonian
\begin{equation}
H=\int_{\mathbb{T}^2}|\nabla u|^2+\frac{1}{3}\int_{\mathbb{T}^2}|u|^6
\end{equation}
having the mass (the $L^2$ norm) and the momentum
\begin{equation}
L=\int_{\mathbb{T}^2}|u|^2\ , \quad M= \int_{\mathbb{T}^2} \Im (u\cdot \nabla u) 
\end{equation}
as constants of motion.
The well-posedness result of \cite{Bou93,BGT} for data $u_0 \in  H^s(\T^2)$, $s\geq1$ gives the existence of a global-in-time smooth solution to \eqref{nls} from smooth initial data and one would like to understand some qualitative properies of solutions.

A fruitful approach to this question is to apply the powerful tools of {\em singular perturbation theory}, such as KAM theory, Birkhoff Normal Form, Arnold diffusion, first developed in order to study finite-dimensional systems. 
 
We are interested in the phenomenon of the growth of Sobolev norms, i.e. we look for solutions 
which initially oscillate only on scales comparable to the spatial period and eventually oscillate on arbitrarily short spatial scales. This is a natural extension of the results in \cite{CKSTT} and \cite{GuaKal} which prove similar results for the cubic NLS. In the strategy of the proof, we follow \cite{CKSTT} as closely as possible; therefore our main result is the precise analogue of the one stated in \cite{CKSTT} for the cubic NLS. Namely, we prove 
\begin{theorem}\label{main.thm}
Let $s>1$, $K\gg 1$ and $ 0<\delta\ll 1$ be given parameters.Then there exists a global smooth solution $u(t,x)$ to \eqref{nls} and a time $T > 0$  with
$$ 
\| u(0)\|_{H^s(\T^2)}\leq \delta\qquad \text{and}\qquad  \| u(T)\|_{H^s(\T^2)}\geq K\,.
$$
\end{theorem}
Note that we are making no claim regarding the time $T$ over which the growth of Sobolev norms occurs, this is the main difference between the approaches of \cite{CKSTT} and \cite{GuaKal}. 

\subsection{Some literature}

The growth of Sobolev norms for solutions of the Non Linear Schr\"odinger equation has been studied widely in the literature, but most of the results regard upper bounds on such growth. In the one dimensional case with an analytic non-linearity
$ \partial_{\bar u} P(|u|^2)$ Bourgain \cite{Bou96} and Staffilani \cite{Sta97} proved at most polynomial growth of Sobolev norms. In the same context Bourgain \cite{Bou00} proved a Nekhoroshev  type theorem for a perturbation of the cubic NLS. Namely, for $s$ large and a typical initial datum $u(0) \in  H^s(T)$ of small size $ \|u(0)\|_s \leq \varepsilon$  he proved
$$\sup_{t\leq T}   \|u(t)\|_s\leq C\varepsilon,\quad |t|<T\,,\quad T\leq \varepsilon^{-A}$$
 with $A = A(s) \to 0$ as $s \to \infty$. 
 Similar upper bounds on the growth have been obtained also for the NLS equation on $\R$ and $\R^2$ as well as on compact manifolds. 

We finally mention the paper \cite{FGL} which discusses the existence of stability regions for the NLS on tori.

Concerning instability results for  the NLS on tori, we mention  the papers by Kuksin 
\cite{Kuk97b} (see related works \cite{Kuk95, Kuk96, Kuk97a, Kuk99}) who studied the growth of Sobolev norms for  the equation
$$
-i\partial_tu+\delta\Delta u=|u|^{2p}u\,,\quad p\in\mathbb N 
$$
and constructed solutions whose Sobolev norms grow by an inverse power of $\delta$. 
Note that  the solutions that he obtains (for $p=2$) correspond to orbits
of equation \eqref{nls} with large initial data.  
A big progress appeared in the paper \cite{CKSTT} where the authors prove Theorem \ref{main.thm} for cubic NLS. Note that the  initial data are {\em small} in $H^s$. Finally the paper \cite{GuaKal} follows the same general strategy of \cite{CKSTT} and constructs orbits whose Sobolev norm grows (by an arbitrary factor) in a time which is polynomial in the growth factor. 
This is done by a careful analysis of the equation and using in a clever way various tools from diffusion in finite dimensional systems.

Note that these results do not imply the existence of solutions with diverging Sobolev norm, nor do they claim that the unstable behavior is typical. Recently, Hani \cite{Hani} has achieved a remarkable progress towards the existence of unbounded Sobolev orbits: for a class of cubic NLS equations with non-polynomial nonlinearity, the combination of a result like Theorem \ref{main.thm} with some clever topological arguments leads to the existence of solutions with diverging Sobolev norm.

Regarding growth of Sobolev norms for other equations we mention the following papers:
\cite{Bou96}-- for the wave equation with a cubic nonlinearity but with a spectrally defined Laplacian,
  \cite{GG10, Poc11}--for the Szeg\"o equation, and \cite{Poc13}--for certain nonlinear wave equations.
  We also mention the long time stability results obtained in  \cite{Bam97, Bam99, Bam03, BG06, GIP09, GKP09, Wang,Wang3}.
  
   A dual point of view to instability is to construct quasi-periodic orbits. These are non-generic solutions which are global in time and whose Sobolev norms are approximately constant. Among the relevant literature we mention \cite{W1,CW,Po2,KP,Bou98,BB1,EK,GYX,BBi10,Wang2,PX,BCP}.
  Of particular interest  are the recent results obtained through KAM theory which gives information on linear stability close to the quasi-periodic solutions.
  In  particular the paper \cite{PP13} proves the existence of  both stable and unstable tori (of arbitrary finite dimension) for the cubic NLS. 
  
  In finite dimensional systems diffusive orbits are usually constructed by proving that the stable and unstable manifolds of a {\em chain of unstable tori} intersect. Usually this is done with tori of co-dimension one so that the manifolds should intersect for dimensional reasons. 
   Unfortunately in the infinite dimensional case one is not able to prove the existence of co-dimension one tori. 
  Actually the construction of almost-periodic orbits is an open problem except for very special cases such as integrable equations or equations with infinitely many external parameters (see for instance \cite{Po,ChP,Bou}). 
  
  In \cite{CKSTT}, \cite{GuaKal} (and the present paper) this problem is avoided  by taking advantage of the specific form of the equation. First one reduces to an approximate equation, i.e. the {\em first order Birkhoff normal form} see \eqref{normalform}. 
  Then for this dynamical system one proves directly the existence of 
    chains of one dimensional unstable tori (periodic orbits) together with their heteroclinic connections.   Next one proves the existence of a {\em slider solution} which shadows the heteroclinic chain in a finite time. Finally, one proves the persistence of the slider solution for the full NLS. In the next section we describe the strategy more in detail.

\subsection{Informal description of the results}
In order to understand the dynamics of \eqref{nls} it is convenient  to pass to a {\em moving frame}  Fourier series representation: $$u(t,x)=\sum_{j\in\mathbb{Z}^2}a_j(t)e^{ij\cdot x+ i |j|^2 t}\,,$$ 
%
so that the equations of motion become
\begin{equation}\label{NLS}
-i\dot a_j=\sum_{j_1,j_2,j_3,j_4,j_5\in\mathbb{Z}^2 \atop j_1+j_2+j_3-j_4-j_5=j}a_{j_1}a_{j_2}a_{j_3}\bar{a}_{j_4}\bar{a}_{j_5}e^{i\omega_6t}
\end{equation}
where $\omega_6 = |j_1|^2+|j_2|^2+|j_3|^2-|j_4|^2-|j_5|^2-|j|^2$.

We define the {\em resonant truncation} of \eqref{NLS} as
\begin{equation}\label{normalform}
-i\dot{\beta}_j=\sum_{\substack{j_1,j_2,j_3,j_4,j_5\in\mathbb{Z}^2\\j_1+j_2+j_3-j_4-j_5=j\\|j_1|^2+|j_2|^2+|j_3|^2-|j_4|^2-|j_5|^2=|j|^2}}\beta_{j_1}\beta_{j_2}\beta_{j_3}\bar{\beta}_{j_4}\bar{\beta}_{j_5}\ .
\end{equation}
It is well known that the dynamics of \eqref{NLS} is well approximated by the one  of \eqref{normalform} for finite but long times\footnote{Actually, passing to the resonant truncation is equivalent to performing the first step of a Birkhoff normal form. However, since we follow closely the proof in \cite{CKSTT}, we chose to use similar notation.}. Our aim is to first prove Theorem \ref{main.thm} for \eqref{normalform} and then extend the result to \eqref{NLS} by an approximation Lemma. The idea of the approximation Lemma roughly speaking is that, by integrating in time the l.h.s. of \eqref{NLS}, one sees that the non-resonant terms (i.e., with $\omega_6\neq0$) give a contribution of order $O(a^9)$. By scaling $a^{(\lambda)}(t)=\lambda^{-1}a(\lambda^{-4}t)$, we see that the non-resonant terms are an arbitrarily small perturbation w.r.t. the resonant terms appearing in \eqref{normalform} and hence they can be ignored for arbitrarily long finite times.

We now outline the strategy used in order to prove Theorem \ref{main.thm} for the equations \eqref{normalform}.

The equations \eqref{normalform} are Hamiltonian with respect to the Hamiltonian function
\begin{equation}\label{nonlin.ham}
\mathcal{H}=\frac{1}{3}\sum_{\substack{j_1,j_2,j_3,j_4,j_5,j_6\in\mathbb{Z}^2\\j_1+j_2+j_3=j_4+j_5+j_6\\|j_1|^2+|j_2|^2+|j_3|^2=|j_4|^2+|j_5|^2+|j_6|^2}}\beta_{j_1}\beta_{j_2}\beta_{j_3}\bar{\beta}_{j_4}\bar{\beta}_{j_5}\bar{\beta}_{j_6}
\end{equation}
and the symplectic form $\Omega=id\beta\wedge d\bar{\beta}$.

This is still a very complicated (infinite-dimensional) Hamiltonian system, but it has the advantage of having many invariant subspaces on which the dynamics simplifies significantly. Let us set up some notation.
\begin{definition}[Resonance]\label{resonance}
A sextuple $(k_1,k_2,k_3,k_4,k_5,k_6)\in(\Z^2)^6$, is a resonance if 
\begin{equation}\label{complete}
k_1+k_2+k_3-k_4-k_5-k_6=0\,,\qquad |k_1|^2+|k_2|^2+|k_3|^2-|k_4|^2-|k_5|^2-|k_6|^2=0\,.
\end{equation}
A resonance is {\em trivial} if it  is of the form $(k_1,k_2,k_3,k_1,k_2,k_3)$ up to permutations of the last three elements.
\end{definition}
\begin{definition}[Completeness]\label{completeness}
We say that a set $\SS\subset\Z^2$ is {\em complete} if the following holds:\\
for every quintuple $(k_1,k_2,k_3,k_4,k_5)\in\SS^5$, if there exists $k_6\in\Z^2$ s.t.
$(k_1,k_2,k_3,k_4,k_5,k_6)$ is a resonance, 
then $k_6\in\SS$.
\end{definition}
It is easily seen that, for any complete $\SS\subset\Z^2$, the subspace defined by requiring $\beta_k = 0$ for all $k\notin\mathcal S$ is invariant.
\begin{definition}[Action preserving]\label{integrable}
A complete set $\SS\subset\Z^2$ is said to be {\em action preserving} if all the resonances in $\SS$ are trivial.
\end{definition}
Remark that, for any complete and action preserving $\SS\subset\Z^2$, the Hamiltonian restricted to $\SS$ is given by (see \cite{ProcesiProcesi})
\begin{equation}
\mathcal{H}|_{\SS}=\frac{1}{3}\left(\sum_{j\in\SS}|\beta_j|^6+9\sum_{\substack{j,k\in\SS\\j\neq k}}|\beta_j|^4|\beta_k|^2+36\sum_{\substack{j,k,m\in\SS\\j\prec k\prec m}}|\beta_j|^2|\beta_k|^2|\beta_m|^2\right)
\end{equation}
where $\preceq$ is any fixed total ordering of $\mathbb{Z}^2$.

If $\SS$ is complete and action preserving, then $\mathcal{H}|_{\SS}$ is function of the actions $|\beta_j|^2$ only, with non-vanishing twist: therefore, the corresponding motion is periodic, quasi-periodic or almost-periodic, depending on the initial data. In particular, if $\beta_j(0)=\beta_k(0)$ for all $j,k\in\SS$, then the motion is periodic. Finally, since all the actions are constants of motion, then so are the $H^s$-norms of the solution.

On the other hand, it is easy to give examples of sets $\SS$ that are complete but not action preserving. For instance, one can consider complete sets of the form $\mathfrak S^{(1)}=\{k_1,k_2,k_3,k_4\}$, where the $k_j$'s are the vertices of a non-degenerate rectangle in $\Z^2$, or of the form $\mathfrak S^{(2)}=\{k_1,k_2,k_3,k_4,k_5,k_6\}$, where the $k_j\in\Z^2$ are all distinct and satisfy equations \eqref{complete}. Other examples are sets of the form $\mathfrak S^{(3)}=\{k_1,k_2,k_3,k_4\}$, with
\begin{equation}\label{SGreTho}
k_1+2k_2-2k_3-k_4=0\,,\qquad |k_1|^2+2|k_2|^2-2|k_3|^2-|k_4|^2=0
\end{equation}
studied in \cite{GreTho} or, more in general, the sets   $\mathfrak S^{(4)}=\cup_j\mathfrak S^{(3)}_j$ studied in  \cite{HauTho}\footnote{The papers \cite{GreTho,HauTho} actually consider the one-dimensional case, but of course the construction of complete sets can always be trivially extended to higher dimensions.}.  In all these cases, the variation of the $H^s$-norm of the solution is of order $O(1)$. Note that, while sets of the form $\mathfrak S^{(2)},\mathfrak S^{(3)},\mathfrak S^{(4)}$ exist in $\Z^d$ for all $d$, the non-degenerate rectangles $\mathfrak S^{(1)}$ exist only in dimension $d\geq2$. Let us briefly describe the dynamics on these sets.  By writing the Hamiltonian in symplectic polar coordinates $\beta_j=\sqrt{I_j}e^{i\theta_j}$, one sees that all these systems are integrable. However, their phase portraits are quite different. In $\mathfrak S^{(1)}$ one can exhibit two periodic orbits $\T_1,\T_2$ that are linked by a heteroclinic connection. $\T_1$ is supported on the modes $k_1,k_2$ and $\T_2$ on $k_3,k_4$. The $H^s$-norm of each periodic orbit is constant in time. By choosing $\mathfrak S^{(1)}$ appropriately one can ensure that these two values are different and this produces a growth of the Sobolev norms. Moreover, all the energy is transferred from $\T_1$ to $\T_2$. In the other cases, i.e. $\mathfrak S^{(2)},\mathfrak S^{(3)},\mathfrak S^{(4)}$, there is no orbit transferring all the energy from some modes to others (see Appendix \ref{disegni}).

These heteroclinic connections are the key to the energy transfer. In fact, assume that
$$
\SS_1:=\left\{\mathtt v_1,\ldots,\mathtt v_n\right\}\,,\qquad\SS_2:=\left\{\mathtt w_1,\ldots,\mathtt w_n\right\}
$$
with $n$ even, are two complete and action preserving sets. Assume moreover that, for all $1\leq j\leq n/2$, $\{\mathtt v_{2j-1},\mathtt v_{2j},\mathtt w_{2j-1},\mathtt w_{2j}\}$ are the vertices of a rectangle as in $\mathfrak S^{(1)}$. Finally, assume that $\SS_1\cup\SS_2$ is complete and contains no  non-trivial resonances except  those of the form $(k,\mathtt v_{2j-1},\mathtt v_{2j},k,\mathtt w_{2j-1},\mathtt w_{2j})$.
As in the case of $\mathfrak S^{(1)}$ the periodic orbits:
$$
\T_1:\quad \beta_{\mathtt v_j}(t)=b_1(t)\neq 0\,, \quad \beta_{\mathtt w_j}(t)=0 \,,\quad \forall j=1,\dots,n
$$
and 
$$
\T_2:  \quad \beta_{\mathtt w_j}(t)=b_2(t)\neq 0\,, \quad \beta_{\mathtt v_j}(t)=0 \,,\quad \forall j=1,\dots,n
$$
are linked by a heteroclinic connection.

We iterate this procedure constructing a {\em generation set} $\mathcal S=\cup_{i=1}^N \mathcal S_i$ where each $\SS_i$ is complete and action preserving. The corresponding periodic orbit $\T_i$ is linked by heteroclinic connections to $\T_{i-1}$ and $\T_{i+1}$. There are two delicate points:
\begin{itemize}
\item[(i)] at each step, when adding a new generation $\SS_i$, we need to ensure that the resulting generation set is still complete and contains no non-trivial resonances except for the prescribed ones. Such resonances are those of the form $(k,v_1,v_2,k,v_3,v_4)$ where $v_1,v_2\in\SS_i$, $v_3,v_4\in\SS_{i+1}$ for some $1\leq i\leq N-1$ and $\{v_1,v_2,v_3,v_4\}$ are the vertices of a rectangle.
\item[(ii)] we need to ensure that the Sobolev norms grow by an arbitrarily large factor $K/\delta$, which requires taking $n$ (the number of elements in each $\SS_j$) and $N$ (the number of generations) large. 
\end{itemize}

The point (i) is a question of combinatorics. It requires some careful classification of the possible resonances and it turns out to be significantly more complicated than the cubic case. We discuss this in subsection \ref{genset}.

The point (ii) is treated exactly in the same way as in \cite{CKSTT}, we discuss it for completeness in subsection \ref{density}, Remark \ref{Nn}.

\smallskip

Given a generation set $\SS$ as above we proceed in the following way:
first we restrict to the finite-dimensional invariant subspace where $\beta_k=0$ for all $k\notin \SS$.
To further simplify the dynamics we restict to the invariant subspace:
$$ \beta_v(t)= b_i(t) \,,\qquad \forall v\in \SS_i\,, \quad \forall i=1,\ldots,n $$ 
this is the so called {\em toy model}. Note that the periodic solutions $\T_i$ live in this subspace. The toy model is a Hamiltonian system, with Hamiltonian given by \eqref{hacca} and with the constant of motion $J=\sum_{i=1}^N|b_i|^2$. We work on the sphere $J=1$, which contains all the $\T_i$ with action $|b_i|^2=1$.

As discussed above, we construct a chain of heteroclinic connections going from $\T_1$ to $\T_N$. Then, we prove (see Proposition \ref{prop.slider}) the existence of a {\em slider solution} which ``shadows'' this chain, starting at time $0$ from a neighborhood of $\T_3$ and\footnote{One could ask why we construct a slider solution diffusing from the third mode $b_3$ to the third-to-last mode $b_{N-2}$, instead of diffusing from the first mode $b_1$ to the last mode $b_N$. The reason is that, since we rely on the proof given in \cite{CKSTT}, our statement is identical to the ones in Proposition 2.2 and Theorem 3.1 in \cite{CKSTT}. As in \cite{CKSTT}, also in our case it would be possible to diffuse from the first to the last mode just by overcoming some very small notational issues.} ending at time $T$ in a neighborhood of $\T_{N-2}$.

 We proceed as follows: first, we perform a symplectic reduction that will allow us to study the local dynamics close to the periodic orbit $\T_j$, which puts the Hamiltonian in form \eqref{reduced}. The new variables $c_k$ are the ones obtained by synchronizing the $b_k$ ($k\neq j$) with the phase of $b_j$. Then, we diagonalize the linear part of the vector field associated to \eqref{reduced}. In particular, the eigenvalues are the Lyapunov exponents of the periodic orbit $\T_j$. As for the cubic case, one obtains that all the eigenvalues are purely imaginary, except four of them which, due to the symmetries of the problem, are of the form $\lambda,\lambda,-\lambda,-\lambda\in\R$. Note that these hyperbolic directions are directly related to the heteroclinic connections connecting $\T_j$ to $\T_{j-1}$ and to $\T_{j+1}$.
%
It turns out that the heteroclinic connections are straight lines in the variables $c_k$. The equations of motion for the reduced system have the form
 \eqref{brut} (which is very similar to the cubic case): this is crucial in order to be able to apply almost verbatim the proof given in \cite{CKSTT}. Note that it is not obvious {\em a priori} that equations \eqref{brut} hold true: for instance, this turns out to be false for the NLS of degree $\geq7$.

The strategy of the proof, which is exactly the same as in \cite{CKSTT}, consists substantially of two parts:
\begin{itemize}
\item studying the linear dynamics close to $\T_j$, treating the non-linear terms as a small perturbation: one needs to prove that the flow associated to equations \eqref{brut} maps points close to the \textit{incoming} heteroclinic connection (from $\T_{j-1}$) to points close to the \textit{outcoming} heteroclinic connection (towards $\T_{j+1}$) (note that, in order to take advantage of the linear dynamics close to $\T_j$, we need that almost all the energy is concentrated on $\SS_j$);
\item following closely the heteroclinic connection in order to flow from a neighborhood of $\T_j$ to a neighborhood of $\T_{j+1}$.
\end{itemize}
The precise statement of these two facts requires the introduction of the notions of {\em targets} and {\em covering} and is summarized in Proposition \ref{covering}. The main analytical tool for the proof are repeated applications of Gronwall's lemma. Our proof of Proposition \ref{covering} follows almost verbatim the proof of the analogue statement, given in Section 3 of \cite{CKSTT}. However, the only way to check that the proof works also in our case, is to go through the whole proof in \cite{CKSTT}, which is rather long and technical, and make the needed adaptations. Therefore, for the convenience of the reader, in Appendix \ref{slider-sol} we give a summary of the proof of Proposition \ref{covering}, highlighting the points where there are significant differences with \cite{CKSTT}.

\subsection{Comparison with the cubic case and higher order NLS equations}

In the cubic NLS, the only resonant sets of frequencies are rectangles, which makes completely natural the choice of using rectangles as building blocks of the generation set $\SS$. In the quintic and higher degree NLS many more resonant sets appear, which a priori gives much more freedom in the construction of $\SS$. In particular, in the quintic case sets of the form $\mathfrak S^{(2)}$ are the most generic resonant sets, and therefore it would look reasonable to use them as building blocks. However (see Appendix \ref{disegni}), such a choice does not allow full energy transfer from a generation to the next one and is therefore incompatible with our strategy. The same happens if one uses sets of the form $\mathfrak S^{(3)}$. This leads us to use rectangles for the construction of $\SS$ also in the quintic case.

It is worth remarking that, while non-degenerate rectangles do not exist in one space dimension, sets of the form $\mathfrak S^{(2)}$, $\mathfrak S^{(3)}$ already exist in one dimension. The equations of the toy model only depend on the combinatorics of the set $\SS$. Therefore, if one were able to prove diffusion in a toy model built with resonant sets of the form $\mathfrak S^{(2)}$, $\mathfrak S^{(3)}$ (or other resonant sets that exist already in one dimension), then one could hope to prove the same type of result for some one-dimensional (non-cubic) NLS.

The use of rectangles as building blocks for the generation set of a quintic or higher order NLS makes things more complicated, since the rectangles induce many different resonant sets, see section \ref{toy.model}. This leads to combinatorial problems which make it harder to prove the non-degeneracy and completeness of $\SS$. The equations of the toy model also have a more complicated form than in the cubic case. Since this type of difficulties grows with the degree, dealing with the general case will most probably require some careful -- and possibly complicated-- combinatorics and one cannot expect to have a completely explicit formula for the toy model Hamiltonian of any degree.

In the quintic case the formula is explicit and relatively simple and we can explicitly perform the symmetry reduction. After some work, we still get equations of the form \eqref{brut}, which resembles the cubic case with some relevant differences: here the Lyapunov exponent $\lambda$ depends on $n$ and tends to infinity as $n\to\infty$; moreover, the non-linear part of the vector field associated to \eqref{reduced} is not homogeneous in the variables $c_k$, as it contains both terms of order 3 and 5 (in the cubic case, it is homogeneous of order 3).

For the NLS of higher degree, not only the reduced Hamiltonian gets essentially unmanageable, but there also appears a further  difficulty. Already for the NLS of degree $7$, a toy model built using rectangles (after symplectic reduction and diagonalization) does not satisfy equations like \eqref{brut}, meaning that the heteroclinic connections are not straight lines. Such a problem can be probably overcome, but this requires a significant adaptation of the analytical techniques used in order to prove the existence of the slider solution (work in progress with M. Guardia).

\subsection{Plan of the paper}

In Section \ref{toy.model} We assume to have a generation set $\SS=\cup_{i=1}^N \SS_i$ which satisfies all the needed non-degeneracy properties and deduce the form of the toy model Hamiltonian.  Then we study this Hamiltonian and prove the existence of slider solutions.

In Section \ref{very.big.s} we prove the existence of  non-degenerate  generation sets such that the corresponding slider solution undergoes the required growth of Sobolev norms.

In Section \ref{all.together.now} we prove, via the approximation Lemma \ref{lemma.approx1} and a scaling argument, the persistence of solutions with growing Sobolev norm for the full NLS.

Since some of the proofs follow very closely the ones in \cite{CKSTT}, we move them to appendix.



\section{The toy model}\label{toy.model}
We now define a finite subset $\mathcal S=\cup_{i=1}^N \mathcal S_i\subset \Z^2$  which satisfies appropriate {\em non-degeneracy conditions} (Definition \ref{def.ND}) as explained in the introduction. In the following we assume that such a set exists. This is not obvious and will be discussed in section \ref{genset}.

 For reasons that will be clear, and following \cite{CKSTT}, the $\SS_i$'s will be called \textit{generations}.
 In order to describe the resonances which connect different generations we introduce some notation. 

\begin{definition}[Family]\label{def.fam}
A family (of age $i\in\{1,\ldots,N-1\}$) is a list $(v_1,v_2;v_3,v_4)$ of elements of $\SS$ such that the points form the vertices of a non-degenerate rectangle
$$
v_1+v_2=v_3+v_4\,,\quad |v_1|^2+|v_2|^2=|v_3|^2+|v_4|^2
$$
and such that one has $v_1,v_2\in\SS_i$ and $v_3,v_4\in\SS_{i+1}$.
Whenever $(v_1,v_2;v_3,v_4)$ form a family, we say that $v_1,v_2$ are the parents of $v_3,v_4$ and that $v_3,v_4$ are the children of $v_1,v_2$. Moreover, we say that $v_1$ is the spouse of $v_2$ (and vice versa) and that $v_3$ is the sibling of $v_4$ (and vice versa). We denote (for instance) $v_1=v_3^{par_1}$, $v_2=v_3^{par_2}$, $v_1=v_2^{sp}$, $v_4=v_3^{sib}$, $v_3=v_1^{ch_1}$, $v_4=v_1^{ch_2}$.
\end{definition}
\begin{remark}
If $(v_1,v_2;v_3,v_4)$ is a family of age $i$, then the same holds for its trivial permutations $(v_2,v_1;v_3,v_4)$, $(v_1,v_2;v_4,v_3)$ and $(v_2,v_1;v_4,v_3)$.
\end{remark}
\begin{definition}\label{def.res}
An integer vector $\lambda\in \Z^{|\mathcal S|}$ such that
$$
\sum_i \lambda_i=0\,,\quad |\lambda|:= \sum_i|\lambda_i|\leq 6 
$$
 is resonant for $\SS$ if
$$ \sum_i \lambda_i v_i=0\,,\quad \sum_i \lambda_i |v_i|^2=0 $$

\end{definition}
Note that to a family $\FF=(v_1,v_2;v_3,v_4)$ we associate a special resonant vector $\lambda^{\FF}$ with $|\lambda|=4$, through $\sum_i \lambda_i^{\FF} v_i=v_1+v_2-v_3-v_4$. Similarly, to the couple of parents in the family $\FF$ we associate the vector $\lambda^{\FF_p}$ through $\sum_i \lambda_i^{\FF_p} v_i=v_1+v_2$ and to the couple of children we associate $\lambda^{\FF_c}$ through $\sum_i \lambda_i^{\FF_c} v_i=v_3+v_4$, so that $\lambda^{\FF}=\lambda^{\FF_p}-\lambda^{\FF_c}$.

\begin{definition}[Generation set]\label{def.gen}
The set $\SS$ is said to be a generation set if it satisfies the following:
\begin{enumerate}
\item For all $i\in\{1,\ldots,N-1\}$, every $v\in\SS_i$ is a member of one and only one (up to trivial permutations) family of age $i$. We denote such a family by $\FF^v$. (Note that $\FF^v=\FF^w$ if $v=w^{sp}$.)
\item For all $i\in\{2,\ldots,N\}$, every $v\in\SS_i$ is a member of one and only one (up to trivial permutations) family of age $i-1$. We denote such a family by $\FF_v$. (Note that $\FF_v=\FF_w$ if $v=w^{sib}$.)
\item For all $v\in\cup_{i=2}^{N-1}\SS_i$, one has $v^{sp}\neq v^{sib}$.
\end{enumerate}
\end{definition}

\begin{remark}
The vectors $\lambda^{\FF}$ corresponding to the families of a generation set are linearly independent.
\end{remark}

Note that, whenever two families $\FF_1$ and $\FF_2$ have a common member (which must be a child in one family and a parent in the other one), then $\lambda^{\FF_1}+\lambda^{\FF_2}$ is a non-trivial resonant vector whose support has cardinality exactly $6$. This motivates the following definition:
\begin{definition}[Resonant vector of type CF]\label{def.CF}
A resonant vector $\lambda$ is said to be of type CF (couple of families) if there exist two families $\FF_1\neq\FF_2$ such that $\lambda=\pm(\lambda^{\FF_1}+\lambda^{\FF_2})$. (Note that, since $|\lambda|\leq6$, the two families $\FF_1,\FF_2$ must have a common member.)
\end{definition}

We say that a generation set is non-degenerate if the following condition is fulfilled.
\begin{definition}[Non-degeneracy]\label{def.ND}
Suppose that there exists $\lambda\in \Z^{|\mathcal S|}$, with $\sum_i\lambda_i=1$ and $|\lambda|\leq5$, such that
$$
\sum_i \lambda_i |v_i|^2-|\sum_i \lambda_i v_i|^2=0\ .
$$
Then only four possibilities are allowed:
\begin{enumerate}
\item $|\lambda|=1$.
\item $|\lambda|=3$ and the support of $\lambda$ consists exactly of three distinct elements of the same family and the two $\lambda_i$'s appearing with a positive sign correspond either to the two parents or to the two children of the family.
\item $|\lambda|=5$ and there exist a family $\FF$ and an element $v\in\SS$ such that $\lambda=\pm\lambda^{\FF}+v$.
\item $|\lambda|=5$ and there exists $v\in\SS$ such that $\lambda-v$ is a resonant vector of type CF.
\end{enumerate}
\end{definition}

Note that, if $\SS$ is a non-degenerate generation set and $\lambda$ is a resonant vector, then either $\lambda=\pm\lambda^{\FF}$ for some family $\FF$ or $\lambda$ is a resonant vector of type CF.

In what follows we will assume that $\SS$  is a non-degenerate generation set. 
This implies that  $\mathcal S$ is complete and all the  subsets $\SS_i$ are pairwise disjoint, complete and action preserving.
Finally the only resonances which appear are those induced by the family relations.
Then, the Hamiltonian restricted to $\SS$ is 
\begin{equation}\label{hamish}
\mathcal{H}\vert_\SS=\frac{1}{3}\left(\sum_{j\in\SS}|\beta_j|^6+9\sum_{\substack{j,k\in\SS\\j\neq k}}|\beta_j|^4|\beta_k|^2+36\sum_{\substack{j,k,m\in\SS\\j\prec k\prec m}}|\beta_j|^2|\beta_k|^2|\beta_m|^2\right) +
\end{equation}
$$
+3\sum_{i=1}^{N-1}\sum_{j\in\SS_i}({\beta_j}\beta_{j^{sp}}\bar\beta_{j^{ch_1}}\bar\beta_{j^{ch_2}}+{\bar\beta_j}\bar\beta_{j^{sp}}\beta_{j^{ch_1}}\beta_{j^{ch_2}})\left(2\sum_{\substack{k\in\SS\\k\notin\FF^j}}|\beta_k|^2+\sum_{m\in\FF^j}|\beta_m|^2\right) +
$$
$$
+12\sum_{i=2}^{N-1}\sum_{j\in\SS_i}\left(\beta_{j^{par_1}}\beta_{j^{par_2}}\beta_{j^{sp}}\bar\beta_{j^{sib}}\bar\beta_{j^{ch_1}}\bar\beta_{j^{ch_2}}+\beta_{j^{ch_1}}\beta_{j^{ch_2}}\beta_{j^{sib}}\bar\beta_{j^{sp}}\bar\beta_{j^{par_1}}\bar\beta_{j^{par_2}}\right)\ .
$$

We restrict to the invariant subspace $\DD\subset\SS$ where $\beta_k= b_i$ for all $k\in \SS_i$ and $\forall i=1,\ldots,N$. Denote by $n$ (which must be an even integer number) the cardinality of each generation. Following the construction in \cite{CKSTT}, one has $n=2^{N-1}$. A straightforward computation (involving some easy combinatorics) of the Hamiltonian yields

\begin{eqnarray*}
\frac{3}{n}\mathcal{H}|_{\DD}=&&\sum_{k=1}^{N}|b_k|^6+9\left[(n-1)\sum_{k=1}^{N}|b_k|^6+n\sum_{\substack{k,\ell=1\\k\neq\ell}}^{N}|b_k|^4|b_{\ell}|^2\right]+\\
&+&6\left[(n-1)(n-2)\sum_{k=1}^{N}|b_k|^6+3n(n-1)\sum_{\substack{k,\ell=1\\k\neq\ell}}^{N}|b_k|^4|b_{\ell}|^2\right]+\\
&+&36n^2\sum_{\substack{k,\ell,m=1\\k<\ell<m}}^{N}|b_k|^2|b_{\ell}|^2|b_m|^2+\\
&+&18\sum_{k=1}^{N-1}\left(-|b_k|^2-|b_{k+1}|^2+n\sum_{\ell=1}^{N}|b_{\ell}|^2\right)(b_k^2\bar{b}_{k+1}^2+b_{k+1}^2\bar{b}_k^2)+\\
&+&36\sum_{k=2}^{N-1}|b_k|^2(b_{k-1}^2\bar{b}_{k+1}^2+b_{k+1}^2\bar{b}_{k-1}^2)
\end{eqnarray*}

The equations of motion for the toy model can be deduced by considering the effective Hamiltonian $h(b,\bar b):=\frac{\mathcal H|_{\DD}(b,\bar b)}{n}$, endowed with the symplectic form $\Omega=idb\wedge d\bar b$.

Due to the conservation of the total mass $L$, the quantity $$J:=\frac{L}{n}=\sum_{k=1}^N|b_k|^2$$ is a constant of motion.

Evidencing the dependence on $J$, we get
\begin{eqnarray}\label{hacca}
3h &= & 4 \sum_{k=1}^{N}|b_k|^6 - 9 n \sum_{h=1}^N|b_h|^2 \left[ \sum_{k=1}^{N}|b_k|^4  -2 \sum_{k=1}^{N-1} (b_k^2\bar{b}_{k+1}^2+b_{k+1}^2\bar{b}_k^2)\right]\\
&+&18\sum_{k=1}^{N-1}(-|b_k|^2-|b_{k+1}|^2)(b_k^2\bar{b}_{k+1}^2+b_{k+1}^2\bar{b}_k^2)+\nonumber\\
&+&36\sum_{k=2}^{N-1}|b_k|^2(b_{k-1}^2\bar{b}_{k+1}^2+b_{k+1}^2\bar{b}_{k-1}^2)\ .\nonumber
\end{eqnarray}


\subsection{Invariant subspaces}
Since $J$ is a constant of motion the dynamics is confined to its level sets. For simplicity, we will restrict to $J=1$, i.e. to
$$
\Sigma:=\{b\in \mathbb C^{N}: \sum_{k=1}^N |b_k|^2=1\}.
$$
All the monomials in the toy model Hamiltonian have even degree in each of the modes $(b_j,\bar b_j)$, which implies that
$$
\mathrm{Supp}(b):=\left\{1\leq j\leq N\middle|b_j\neq0\right\}
$$
is invariant in time. This automatically produces many invariant subspaces some of which  will play a specially important role, namely:
{\it (i)}  the subspaces $M_j$ corresponding to $\mathrm{Supp}(b)=\{j\}$ for some $1\leq j\leq N$. In this case the dynamics is confined to the circle $|b_j|^2=J$, with
\begin{equation}\label{vvv}
b_j(t)=\sqrt{J}\mathrm{exp}\left[-i\left(3n-\frac43\right)J^2 t\right]\ .
\end{equation}
The intersection of $M_j$ with $\Sigma$
 is a single periodic orbit, which we denote by $\T_j$.

{\it (ii)} The subspaces generated by $M_j$ and $M_{j+1}$ (corresponding to $\mathrm{Supp}(b)=\{j,j+1\}$) for some $1\leq j\leq N-1$. Here the Hamiltonian becomes
\begin{eqnarray}\label{www}
3h_{2g} &= & 4 \left(|b_j|^6+|b_{j+1}|^6\right) - 9 n \left(|b_j|^2+|b_{j+1}|^2\right) \left[|b_j|^4+|b_{j+1}|^4      - 2 (b_j^2\bar{b}_{j+1}^2+b_{j+1}^2\bar{b}_j^2)\right]\nonumber\\
&-&18\left(|b_j|^2+|b_{j+1}|^2\right)\left(b_j^2\bar{b}_{j+1}^2+b_{j+1}^2\bar{b}_j^2\right)
\end{eqnarray}
Passing to symplectic polar coordinates:
$$ b_j= \sqrt{I_1}e^{i \theta_1} \,,\quad b_{j+1}= \sqrt{I_2}e^{i \theta_2} \,,$$ we have
$$
3h_{2g}= (4-9n)(I_1+I_2)^3 +6(I_1+I_2)I_1I_2\left( 3 n-2 + 6 (n-1) \cos(2(\theta_1-\theta_2))  \right)\,,
$$
since $J=I_1+I_2$ is a conserved quantity the dynamics is integrable and easy to study. 

We pass to the symplectic variables
 $$
 J,I_1,\theta_2,\varphi=\theta_2-\theta_1
 $$
 and obtain the Hamiltonian
 $$
 3h_{2g}= (4-9n)J^3 +6JI_1(J-I_1)\left( 3 n-2 + 6 (n-1) \cos(2\varphi)  \right)\,.
 $$
The phase portrait (ignoring the evolution of the cyclic variable $\theta_2$) restricted to $\Sigma$ is described in Figure \ref{fig1}.
 \begin{figure}[!ht]
\centering
\begin{minipage}[b]{11cm}
\centering
{\psfrag{I}{$I_1$}
\psfrag{J}{$I_1=1$}
\psfrag{a}[c]{$-\pi$}
\psfrag{b}[c]{$-\pi+\varphi_0$}
\psfrag{c}[c]{$-\varphi_0$}
\psfrag{f}[c]{$\pi$}
\psfrag{d}[c]{$\varphi_0$}
\psfrag{e}[c]{$\pi-\varphi_0$}
\psfrag{g}{$\varphi$}
\includegraphics[width=11cm]{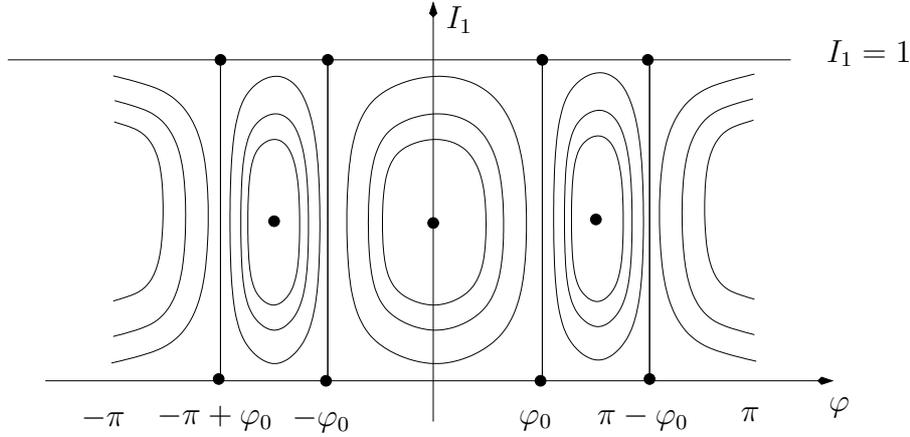}
}
\caption{\footnotesize{Phase portrait of the two generation Hamiltonian $h_{2g}$ on $\Sigma$.}}\label{fig1}
\end{minipage}
\end{figure}

\begin{remark}
The coordinates $I_1,\varphi$ and the domain given by the cylinder $(\varphi,I_1)\in\mathbb S^1\times[0,1]$ are singular since the angle $\varphi=\theta_2-\theta_1$ is ill-defined when $I_1=0$ or $I_1=1$. In the correct picture for the reduced dynamics, each of the lines $I_1=0$ and $I_1=1$ should be shrunk to a single point, thus obtaining (topologically) a two dimensional sphere (see Figure \ref{sphere}).

This can also be seen in the following way. The level set $J=1$ is a three dimensional sphere $\mathbb S^3$, with the gauge symmetry group $\mathbb S^1$ acting freely on it. Due to the Hopf fibration, the topology of the quotient space is $\mathbb S^2$.

 \begin{figure}[!ht]
\centering
\begin{minipage}[b]{11cm}
\centering
\includegraphics[width=6cm]{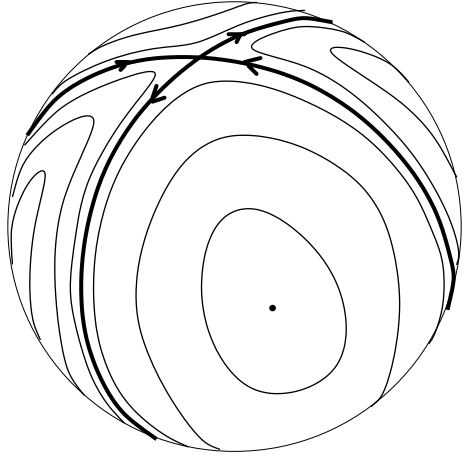}
\caption{\footnotesize{A sketch of the phase portrait of the two generation Hamiltonian $h_{2g}$ on $\Sigma$ in the correct topology.}}\label{sphere}
\end{minipage}
\end{figure}
\end{remark}

 Note that,  as for the case of the cubic NLS (see \cite{CKSTT,GuaKal}), there exist heteroclinic connections linking  $\T_j$ to  $\T_{j+1}$.
 Again as in the cubic case the orbits have fixed angle 
 $$\varphi(t)=\varphi_0= \frac12\arccos(-\frac{3n-2}{6(n-1)})  \,,\quad  I_1(t)= \frac{e^{2\lambda t}}{1+e^{2\lambda t}} $$
 where $\lambda= 2\sqrt{(9n-8)(3n-4)}$.
Our aim will be to construct slider solutions that are very concentrated on the mode $b_3$ at the time $t=0$ and very concentrated on the mode $b_{N-2}$ at the time $t=T$. 
These solutions will start very close to the periodic orbit $\T_3$ and then use the heteroclinic connections in order to slide from $\T_3$ to $\T_4$ and so on until $\T_{N-2}$.

\subsection{Symplectic reduction}

Now, since we are interested in studying the dynamics close to the $j$-th periodic orbit $\T_j$, we introduce a set of coordinates which are in phase with it  and give a symplectic reduction with respect to the constant of motion $J$. This procedure is the same that was carried out, for the cubic NLS, in \cite{GuaKal} and, substantially, already in \cite{CKSTT}.

Let $\vartheta^{(j)}$ be the phase of the complex number $b_j$. Then, for $k\neq j$, let $c_k^{(j)}$ the variable obtained by conjugating $b_k$ with the phase $\vartheta^{(j)}$, i.e. $$c_k^{(j)}=b_ke^{-i\vartheta^{(j)}}\ .$$
Then, the change of coordinates (well defined on $\{b_j\neq 0\}$) given by $$(b_1,\ldots,b_N,\bar b_1,\ldots,\bar b_N)\mapsto$$ $$\mapsto(c_1^{(j)},\ldots,c_{j-1}^{(j)},J,c_{j+1}^{(j)},\ldots,c_N^{(j)},\bar c_1^{(j)},\ldots,\bar c_{j-1}^{(j)},\vartheta^{(j)},\bar c_{j+1}^{(j)},\ldots,\bar c_N^{(j)})$$ is symplectic. Namely, in the new coordinates the symplectic form is given by
$$\Omega=idc^{(j)}\wedge d\bar c^{(j)}+dJ\wedge d\vartheta^{(j)}\ .$$
Then, we rewrite the Hamiltonian $h$ in terms of the new coordinates (from now on, in order to simplify the notation, we will omit the superscript $(j)$ in the $c^{(j)}$ variables and in their complex conjugates $\bar{c}^{(j)}$ and in the phase $\vartheta^{(j)}$). Thus, we get the expression

\begin{eqnarray}\label{reduced}
& 3h & \!=\;  4\sum_{k\neq j}|c_k|^6- 4(\sum_{k\neq j}|c_k|^2)^3+ (18 n- 12) J^2 \sum_{k\neq j}|c_k|^2-9 n J \sum_{k\neq j}|c_k|^4+\nonumber\\&-&  (9 n-12)J(\sum_{k\neq j}|c_k|^2)^2+ 18\sum_{\substack{k=1\\k\neq j-1,j}}^{N-1}\left(-|c_k|^2-|c_{k+1}|^2+nJ\right)(c_k^2\bar{c}_{k+1}^2+c_{k+1}^2\bar{c}_k^2)+\nonumber\\
&+&18\left[\sum_{\substack{k=1\\k\neq j-1,j}}^{N}|c_k|^2+(n-1)J\right]\left(J-\sum_{\substack{\ell=1\\ \ell\neq j}}^{N}|c_{\ell}|^2\right)(c_{j-1}^2+\bar{c}_{j-1}^2)+\nonumber\\
&+&18\left[\sum_{\substack{k=1\\k\neq j,j+1}}^{N}|c_k|^2+(n-1)J\right]\left(J-\sum_{\substack{\ell=1\\ \ell\neq j}}^{N}|c_{\ell}|^2\right)(c_{j+1}^2+\bar{c}_{j+1}^2)+\label{xxx}\\
&+& 36\sum_{\substack{k=2\\k\neq j-1,j,j+1}}^{N-1}|c_k|^2(c_{k-1}^2\bar{c}_{k+1}^2+c_{k+1}^2\bar{c}_{k-1}^2)+
36|c_{j-1}|^2\left(J-\sum_{\substack{k=1\\k\neq j}}^{N}|c_k|^2\right)(c_{j-2}^2+\bar{c}_{j-2}^2)+\nonumber\\
&+& 36\left(J-\sum_{\substack{k=1\nonumber\\k\neq j}}^{N}|c_k|^2\right)(c_{j-1}^2\bar{c}_{j+1}^2+c_{j+1}^2\bar{c}_{j-1}^2)+36|c_{j+1}|^2\left(J-\sum_{\substack{k=1\\k\neq j}}^{N}|c_k|^2\right)(c_{j+2}^2+\bar{c}_{j+2}^2)\ .\nonumber
\end{eqnarray}
Observe that the Hamiltonian $h$ does not depend on $\vartheta$. Since $J$ is a constant of motion, the terms depending only on $J$ can be erased from the Hamiltonian. Up to those constant terms, one has
\begin{equation}
h=h_2+r_4\ ,
\end{equation}
where $h_2$ is the part of order 2 in $(c,\bar c)$ (which corresponds to the linear part of the vector field) and $r_4$ is of order at least 4 in $(c,\bar c)$. By an explicit computation, one obtains
\begin{equation}
h_2=2J^2\left[(3n-2)\sum_{\substack{k=1\\k\neq j}}^{N}|c_k|^2+3(n-1)(c_{j-1}^2+\bar{c}_{j-1}^2+c_{j+1}^2+\bar{c}_{j+1}^2)\right]\ .
\end{equation}
It is easily seen that the dynamics associated to the vector field generated by $h_2$ is elliptic in the modes $c_k$ with $1\leq k\leq j-2$ or $j+2\leq k\leq N$, while it is hyperbolic in the modes $c_{j-1}$ and $c_{j+1}$. In order to put in evidence the hyperbolic dynamics, we perform a change of coordinates which diagonalizes the linear part of the vector field. Namely, for $k=j-1,j+1$, we set
$$c_k=\frac{1}{\sqrt{2\Im(\omega^2)}}(\bar\omega c^-_k+\omega c^+_k)$$
$$\bar{c}_k=\frac{1}{\sqrt{2\Im(\omega^2)}}(\omega c^-_k+\bar\omega c^+_k)$$
where $\omega=e^{i\varphi_0}$ with
$$\varphi_0=\frac{1}{2}\arccos\left(-\frac{3n-2}{6(n-1)}\right)\ .$$

Note that this change of variables affects only the hyperbolic modes, which are expressed in terms of the new variables $(c^+_{j-1},c^-_{j-1},c^+_{j+1},c^-_{j+1})$. This transformation is symplectic, 
 writing $h_2$ as a function of the new variables we get
\begin{equation}
h_2=2J^2\left[(3n-2)\sum_{\substack{k=1\\k\neq j-1,j,j+1}}^{N}|c_k|^2+\sqrt{(9n-8)(3n-4)}(c^+_{j-1}c^-_{j-1}+c^+_{j+1}c^-_{j+1})\right]\ .
\end{equation}
We have proved that the periodic orbit \eqref{vvv} is hyperbolic and we have explicitly written the quadratic part of the Hamiltonian in the local variables. 
Similarly to the case of the cubic NLS these local variables are are actually well adapted to describing also the global dynamics  connecting two periodic orbits, as discussed in the previous section. 

To this purpose we study the integrable {\em two generation} Hamiltonian \eqref{www} after all the changes of variables described in this section, i.e. in the variables $c^+_{j+1},c^-_{j+1}$.  Direct substitution shows that the Hamiltonian is given by
$$h_{2g}=2J\sqrt{(9n-8)(3n-4)}c^+_{j+1}c^-_{j+1}\left\{J-\frac{1}{2\Im(\omega^2)}[(c^+_{j+1})^2+(c^-_{j+1})^2+2\Re(\omega^2)c^+_{j+1}c^-_{j+1}]\right\}\ .$$
It is important to note that all the monomials in $h_{2g}$ contain both $c^+_{j+1}$ and $c^-_{j+1}$, so the subspaces $c^+_{j+1}=0$ and $c^-_{j+1}=0$ (which correspond to the heteroclinic connections) are invariant for the 2-generation dynamics.
It is useful to denote by $c^*=\{c_h\}_{h\neq j-1,j,j+1 }$ so that the dynamical variables of the Hamiltonian \eqref{xxx} become $(c^+_{j-1},c^-_{j-1},c^+_{j+1},c^-_{j+1},c^*,\bar{c}^*)$.

Now, since $$h_{2g}=h|_{c^+_{j-1}=c^-_{j-1}=q_1=0,c^*=0}\ ,$$ exploiting also the symmetry between $(c^+_{j-1},c^-_{j-1})$ and $(c^+_{j+1},c^-_{j+1})$, this implies that also in $h$ none of the monomials in $(c^+_{j-1},c^-_{j-1},c^+_{j+1},c^-_{j+1},c^*,\bar{c}^*)$depends only on one of the variables $c^+_{j-1},c^-_{j-1},c^+_{j+1},c^-_{j+1}$.

Finally, we recall that all the monomials in $h(c^+_{j-1},c^-_{j-1},c^+_{j+1},c^-_{j+1},c^*)$ have even degree in each of the couples $(c^*_k,\bar{c}^*_k)$ and in both couples $(c^+_k,c^-_k)$.

From these observations, and from the bound $O(c^2)\lesssim J=O(1)$, we immediately deduce the following relations about the Hamilton equations associated to $h$:
\begin{eqnarray}
\dot c^-_{j-1} &= &-2J^2\sqrt{(9n-8)(3n-4)}c^-_{j-1}+O(c^2c^-_{j-1})+O(c^2_{\neq j-1}c^+_{j-1})\label{brut}\\
\dot c^+_{j-1} &= &2J^2\sqrt{(9n-8)(3n-4)}c^+_{j-1}+O(c^2c^+_{j-1})+O(c^2_{\neq j-1}c^-_{j-1})\nonumber\\
\dot c^-_{j+1}&= &-2J^2\sqrt{(9n-8)(3n-4)}c^-_{j+1}+O(c^2c^-_{j+1})+O(c^2_{\neq j+1}c^+_{j+1})\nonumber\\
\dot c^+_{j+1} &= & 2J^2\sqrt{(9n-8)(3n-4)}c^+_{j+1}+O(c^2c^+_{j+1})+O(c^2_{\neq j+1}c^-_{j+1})\nonumber\\
\dot{c}^* &= & 2 J^2(3n+2) i c^*+O(c^2c^*)\ ,\nonumber
\end{eqnarray}
where we denote $c=(c^+_{j-1},c^-_{j-1},c^+_{j+1},c^-_{j+1},c^*)$, $c_{\neq j-1}=(c^+_{j+1},c^-_{j+1},c^*)$, $c_{\neq j+1}=(c^+_{j-1},c^-_{j-1},c^*)$. These relations are the precise analogue of Proposition 3.1 in \cite{CKSTT}, where the factor $2J^2\sqrt{(9n-8)(3n-4)}$ here replaces the factor $\sqrt{3}$ in \cite{CKSTT}. 

From the equations of motion \eqref{brut}, we deduce that
$$-i\dot{c}_{j+1}=\frac{\partial h_{2g}}{\partial\bar{c}_{j+1}}+O(c_{j+1}c^2_{\neq j+1})\ .$$
We have
$$h_{2g}=2J\sqrt{(9n-8)(3n-4)}c^+_{j+1}c^-_{j+1}(J-|c_{j+1}|^2)$$
where $c^+_{j+1},c^-_{j+1}$ can be thought of as functions of $(c_{j+1},\bar{c}_{j+1})$. Then
$$\dot{c}_{j+1}=2iJ\sqrt{(9n-8)(3n-4)}(J-|c_{j+1}|^2)\frac{\partial(c^+_{j+1}c^-_{j+1})}{\partial\bar{c}_{j+1}}+O(c_{j+1}c^+_{j+1}c^-_{j+1})+O(c_{j+1}c^2_{\neq j+1})\ .$$
We compute
$$2i\frac{\partial(c^+_{j+1}c^-_{j+1})}{\partial\bar{c}_{j+1}}=\sqrt{\frac{2}{\Im(\omega^2)}}(\omega c^+_{j+1}-\bar{\omega}c^-_{j+1})$$
from which we deduce
\begin{equation}\label{cacca1}
\dot{c}_{j+1}=J\sqrt{\frac{2(9n-8)(3n-4)}{\Im(\omega^2)}}(\omega c^+_{j+1}-\bar{\omega}c^-_{j+1})(J-|c_{j+1}|^2)+O(c_{j+1}c^+_{j+1}c^-_{j+1})+O(c_{j+1}c^2_{\neq j+1})
\end{equation}
which is the analogue for $c_{j+1}$ of equation (3.19) in \cite{CKSTT}. In the same way, one deduces
\begin{equation}\label{cacca2}
\dot{c}_{j-1}=J\sqrt{\frac{2(9n-8)(3n-4)}{\Im(\omega^2)}}(\omega c^+_{j-1}-\bar{\omega}c^-_{j-1})(J-|c_{j-1}|^2)+O(c_{j-1}c^+_{j-1}c^-_{j-1})+O(c_{j-1}c^2_{\neq j-1})
\end{equation}
which is the analogue of equation (3.19) in \cite{CKSTT} for the evolution of $c_{j-1}$.
\subsection{Existence of a ``slider solution''}

In this section, we are going to prove the following proposition (which is the analogue of Proposition 2.2 in \cite{CKSTT}), that establishes the existence of a slider solution.

\begin{proposition}\label{prop.slider}
For all $\epsilon>0$ and $N\geq6$, there exist a time $T_0>0$ and an orbit of the toy model such that
$$|b_3(0)|\geq1-\epsilon,\qquad\qquad|b_j(0)|\leq\epsilon\qquad j\neq3$$
$$|b_{N-2}(T_0)|\geq1-\epsilon,\qquad\qquad|b_j(T_0)|\leq\epsilon\qquad j\neq N-2\ .$$
Furthermore, one has $\|b(t)\|_{\ell^{\infty}}\sim1$ for all $t\in[0,T_0]$.

More precisely there exists a point $x_3$ within $O(\epsilon)$ of $\T_3$ (using the usual metric on $\Sigma$), a point $x_{N-2}$ within $O(\epsilon)$ of $\T_{N-2}$ and a time $T_0\geq 0$ such that $S(T_0)x_3 = x_{N-2}$, where $S(t)x$ is the dynamics at time $t$ of the toy model Hamiltonian with initial datum $x$.
\end{proposition}
In order to prove Proposition \ref{prop.slider} of our paper, we completely rely on the proof of the analogue Proposition 2.2 in \cite{CKSTT}. 
In order to keep our notations as close as possible to those of \cite{CKSTT} we rescale the time $t=2\sqrt{(9n-8)(n-4/3)} \tau$ in our toy model, this means rescaling $h\rightarrow \sqrt{3}h/2\sqrt{(9n-8)(3n-4)} $, where $h$ is defined in \eqref{hacca}, so that the Lyapunov exponents of the linear dynamics are $\sqrt 3$.
We hence prove Proposition \ref{prop.slider} for the rescaled toy model.
By formul\ae \eqref{brut}, \eqref{cacca1}, \eqref{cacca2} we have the analogue of Proposition 3.1 and of eq. (3.19) of \cite{CKSTT}.
\begin{proposition}
Let $3\leq j\leq N-2$ and let $b(\tau)$ be a solution  of the rescaled toy model living on $\Sigma$ and  with $b_j(\tau)\neq 0$. We have the system of equations:
\begin{subequations}\label{bruttone}
\begin{align}
\dot c^-_{j-1} &= -\sqrt{3}c^-_{j-1}+O(c^2c^-_{j-1})+O(c^2_{\neq j-1}c^+_{j-1})\label{bruttone1}\\
\dot c^+_{j-1} &= \sqrt{3} c^+_{j-1}+O(c^2c^+_{j-1})+O(c^2_{\neq j-1}c^-_{j-1})\label{bruttone2}\\
\dot c^-_{j+1} &= -\sqrt{3}c^-_{j+1}+O(c^2c^-_{j+1})+O(c^2_{\neq j+1}c^+_{j+1})\label{bruttone3}\\
\dot c^+_{j+1} &=  \sqrt{3}c^+_{j+1}+O(c^2c^+_{j+1})+O(c^2_{\neq j+1}c^-_{j+1})\label{bruttone4}\\
\dot{c}^* &=  i \kappa  c^*+O(c^2c^*)\ ,\quad \kappa= \frac{\sqrt{3}(3n-2)}{\sqrt{(9n-8)(3n-4)} }\label{bruttone5}
\end{align}
\end{subequations}
Moreover,
\begin{equation}\label{cacca3}
\dot{c}_{j+1}=\sqrt{\frac{3}{2\Im(\omega^2)}}(\omega c^+_{j+1}-\bar{\omega}c^-_{j+1})(J-|c_{j+1}|^2)+O(c_{j+1}c^+_{j+1}c^-_{j+1})+O(c_{j+1}c^2_{\neq j+1})
\end{equation}
and
\begin{equation}\label{cacca4}
\dot{c}_{j-1}=\sqrt{\frac{3}{2\Im(\omega^2)}}(\omega c^+_{j-1}-\bar{\omega}c^-_{j-1})(J-|c_{j-1}|^2)+O(c_{j-1}c^+_{j-1}c^-_{j-1})+O(c_{j-1}c^2_{\neq j-1})
\end{equation}
Finally, since the equations \eqref{bruttone}  come from the Hamitonian \eqref{xxx} which is an even  polynomial of degree six, one has that all the symbols $O(c^3)$ are actually\footnote{ as in \cite{CKSTT}, we use the {\em schematic notation} $\OO(\cdot)$. The symbol $\OO(y)$ indicates a linear combination of terms that resemble $y$ up to the presence of multiplicative constants and complex conjugations. So for instance a term like $2i\bar c_{j+1}|c_{j+2}|^2c_{j+3}^2-3c_{j+1}|c_{j+2}|^4$ is of the form $\OO(c^5)$ and more precisely $\OO(c_{j+1}c_{\neq j+1}^4)$}  $\OO(c^3)+\OO(c^5)$. For instance
\begin{equation}\label{bruttine}
O(c^2c^-_{j-1})= \OO(c^2c^-_{j-1})+ \OO(c^4c^-_{j-1}) \,,\quad O(c^2_{\neq j-1}c^+_{j-1})=\OO(c^2_{\neq j-1}c^+_{j-1})+ \OO(c^2c^2_{\neq j-1}c^+_{j-1})
\end{equation}

\end{proposition}
Note that the only difference with \cite{CKSTT} is that our remainder terms (of type $O(c^2c^-_{j-1})$, $O(c^2_{\neq j-1}c^+_{j-1})$, etc.) are not homogeneous of degree three but have also a term of degree five (which is completely irrelevant in the analysis).

We now introduce some definitions and notations of \cite{CKSTT}. 
\begin{definition}[Targets] A target is a triple $(M,d,R)$, where M is a subset of $\Sigma$, $d$ is a semi-metric on $\Sigma$ and $R>0$ is a radius. We say that a point $x\in \Sigma$  is within a target $(M,d,R)$ if we have $d(x,y)<R$ for some $y\in M$. Given two points $x,y \in \Sigma$, we say that $x$ hits $y$, and write $x \mapsto y$, if we have $y = S(t)x$ for some $t \geq 0$. Given an initial target $(M_1,d_1,R_1)$ and a final target $(M_2,d_2,R_2)$, we say that $(M_1,d_1,R_1)$ can cover $(M_2,d_2,R_2)$, and write $(M_1, d_1, R_1) \twoheadrightarrow (M_2, d_2, R_2)$, if for every $x_2 \in M_2$ there exists an $x_1 \in M_1$, such that for any point $y_1 \in \Sigma$ with $d(x_1, y_1) < R_1$ there exists a point $y_2 \in \Sigma$ with $d_2(x_2, y_2) < R_2$ such that $y_1$ hits $y_2$.
\end{definition}
We refer the reader to pages 64-66 of \cite{CKSTT} for a presentation of the main properites of targets.

 We need a number of parameters. First,  an increasing set of exponents
$$1\ll A^0_3 \ll A^+_3 \ll A^-_4 \ll\cdots\ll A_{N-2}^- \ll A^0_{N-2}.$$
For sake of concreteness, we will take these to be consecutive powers of $10$.
Next, we shall need a small parameter $0< \sigma\ll 1$
depending on $N$ and the exponents $A$ which   basically measures the distance to $\T_j$ at which the quadratic Hamiltonian  dominates the quartic terms.
Then we  need a set of scale parameters $$1\ll r_{N-2}^0\ll r_{N-2}^-\ll r_{N-2}^+\ll r_{N-3}^-\ll\cdots\ll r_3^+ \ll r_3^0$$
where each parameter is assumed to be sufficiently large depending on the preceding parameters and on $\sigma$ and the $A$' s. these parameters represent a certain shrinking of each target from the previous one (in order to guarantee that each target can be covered by the previous).
Finally, we need a very large time parameter $T\gg 1$
that we shall assume to be as large as necessary depending on all the previous parameters.

 Setting $$ \{c_1,\dots,c_{h}\}:=c_{\leq h}\,,\quad \{c_{h},c_{h+1},\dots,c_{N}\}=c_{\geq h}$$
 we call $c_{\leq j-1}$ the {\em trailing  modes},  $c_{\geq j+1}$ the {\em leading  modes} $c_{\leq j-2}$ the {\em trailing peripheral modes} and finally $c_{\geq j+2}$ the {\em leading peripheral modes}.
 We  construct a series of targets:
 \begin{itemize}
 \item An incoming target $(M_j^-,d_j^-ˆ',R_j^-ˆ')$ (located near the stable manifold of $\T_j$) defined as follows:
 
 \noindent $M_j^-$ is the subset of $\Sigma$ where 
 $$
 c_{\leq j-2},c^+_{j-1}=0\,,\quad c^-_{j-1}=\sigma\,,\quad |c_{\geq j+1}|\leq r^-_j e^{-2\sqrt 3 T}\,,
 $$
 $R_j^-=T^{A_j^-}$ and the semi-metric is
 $$
 d_j^- ( x , \tilde x ) : = e^{2\sqrt{3}T} |c_{\leq j-2}-\tilde c_{\leq j-2}|+e^{\sqrt{3}T}  |c_{ j-1}^--\tilde c_{ j-1}^-|+
 $$
 $$
 e^{4\sqrt{3}T}  |c_{ j-1}^++\tilde c_{ j-1}^+|+
 e^{3\sqrt{3}T} |c_{ \geq j+1}-\tilde c_{ \geq j+1}|
 $$
\item A ricochet target $(M_j^0, d_j^0, R_j^0)$ (located very near $\T_j$ itself), defined as follows:
 
 \noindent $M_j^0$ is the subset of $\Sigma$ where 
 $$
 c_{\leq j-1},c^-_{j+1}=0\,,\quad |c^+_{j+1}|\leq r^0_j e^{-\sqrt 3 T}\,,\quad |c_{\geq j+2}|\leq r^0_j e^{-2\sqrt 3 T}\,,
 $$
 $R_j^-=T^{A_j^0}$ and the semi-metric is
 $$
 d_j^0 ( x , \tilde x) : = e^{2\sqrt{3}T} \left(|c_{\leq j-2}-\tilde c_{\leq j-2}|+ |c_{ j+1}^++\tilde c_{ j+1}^+|\right)+e^{\sqrt{3}T}  |c_{ j-1}^--\tilde c_{ j-1}^-|+
 $$
 $$
 e^{3\sqrt{3}T}  \left(|c_{ j-1}^++\tilde c_{ j-1}^+|+|c_{ j+1}^-+\tilde c_{ j+1}^-|+ |c_{ \geq j+2}-\tilde c_{ \geq j+2}|\right)
 $$
\item  An outgoing target $(M_j^+,d_j^+,R_j^+)$ (located near the unstable manifold of $\T_j$) defined as follows:
 
 \noindent $M_j^+$ is the subset of $\Sigma$ where 
 $$
 c_{\leq j-1},c^-_{j+1}=0\,,\quad c^+_{j+1}=\sigma\,,\quad |c_{\geq j+2}|\leq r^+_j e^{-2\sqrt 3 T}\,,
 $$
 $R_j^-=T^{A_j^+}$ and the semi-metric is
 $$
 d_j^+ ( x , \tilde x ) : = e^{2\sqrt{3}T} |c_{\leq j-1}-\tilde c_{\leq j-1}|+e^{4\sqrt{3}T}  |c_{ j+1}^--\tilde c_{ j+1}^-|+
 $$
 $$
 e^{\sqrt{3}T}  |c_{ j+1}^++\tilde c_{ j+1}^+|+
 e^{3\sqrt{3}T} |c_{ \geq j+2}-\tilde c_{ \geq j+2}|
 $$.
\end{itemize}
By Section 3.5 of \cite{CKSTT} Proposition \ref{prop.slider} follows from 
\begin{proposition}\label{covering}
$(M_j^-,d_j^-,R_j^-)\twoheadrightarrow (M_j^0,d_j^0,R_j^0)$ for all $3 < j \leq  N-2$, $ (M_j^0,d_j^0,R_j^0)\twoheadrightarrow(M_j^+,d_j^+,R_j^+)$  for all $3 \leq  j < N-2$,
$(M_j^+,d_j^+,R_j^+)\twoheadrightarrow (M_{j+1}^- ,d_{j+1}^- ,R_{j+1}^-)$ for all $3\leq j <N-2$.
\end{proposition}
\begin{proof}
See Appendix \ref{slider-sol}.
\end{proof}
\begin{proof}[Proof of Proposition \ref{prop.slider}]
By \cite{CKSTT} Lemma 3.1  we deduce the covering relations 
\begin{equation}\label{ceppa}
( M_3^0 , d_3^0 , R_3^0 ) \twoheadrightarrow ( M_{N-2}^0 , d_{N-2}^0, R_{N-2}^0 ),
\end{equation} in turn this implies that  there is at least one solution $b(t)$ to (3.1) which starts within the ricochet target $(M_3^0,d_3^0,R_3^0)$ at some time $t_0$ and ends up within the ricochet target $( M_{N-2}^0 , d_{N-2}^0, R_{N-2}^0 )$  at some later time $t_1 > t_0$. But from the definition of these targets, we thus see that $b(t_0)$ lies within a distance $O(r_3^0e^{-\sqrt3 T} )$ of $\T_3$, while $b(t_1)$ lies within a distance $O(r_{N-2}^0e^{-\sqrt3 T} )$ of $\T_{N-2}$. The claim follows.
\end{proof}

\section{Construction of the set $\SS$}\label{very.big.s}

\subsection{The density argument and the norm explosion property}
\label{density}

The perturbative argument for the construction of the frequency set $\mathcal S$ works exactly as in \cite{CKSTT}, Section 4. However, for the convenience of the reader, we recall here the main points.

A convenient way to construct a generation set is to first fix a ``genealogical tree'', i.e. an abstract combinatorial model of the parenthood and brotherhood relations, and then to choose a placement function, embedding this abstract combinatorial model in $\R^2$. Our choice of the abstract combinatorial model is  the one described in \cite{CKSTT} pp. $99$-$100$.
 \begin{figure}[!ht]
\centering
\begin{minipage}[b]{11cm}
\centering
{\psfrag{a}{\small {first generation}}
\psfrag{b}{\small second generation}
\psfrag{c}{\small third generation}
\psfrag{d}{\small fourth generation}
\psfrag{e}{\small fifth generation}
\psfrag{f}[c]{$4$}
\psfrag{i}[c]{$2$}
\psfrag{h}[c]{$3$}
\psfrag{g}[c]{$6$}
\includegraphics[width=11cm]{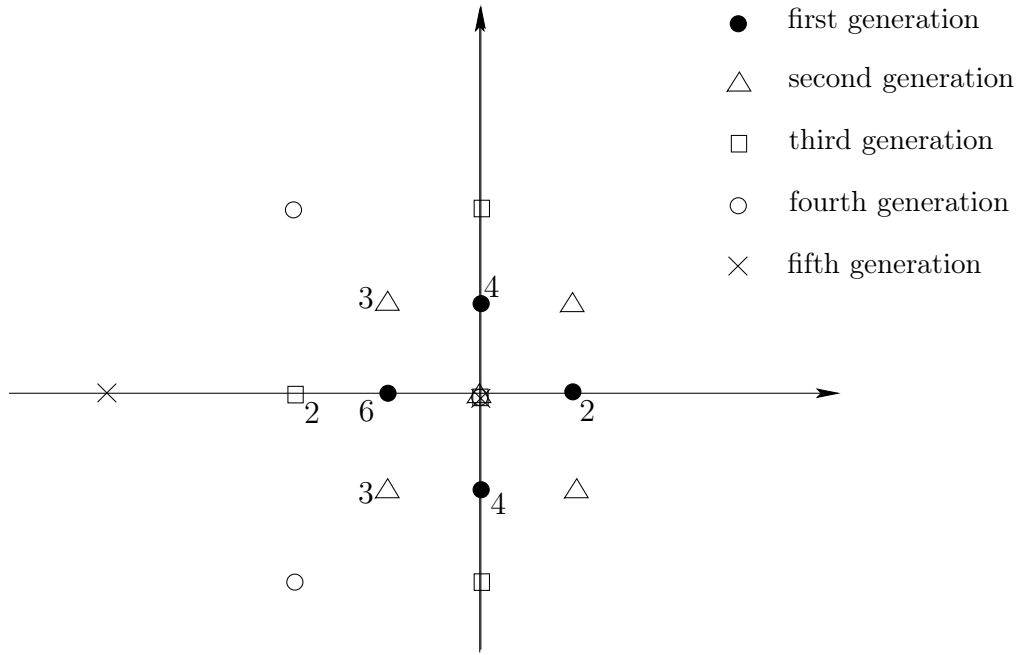}
}
\caption{\footnotesize{The prototype embedding with five generations. Note that this is a highly degenerate realization of  the abstract combinatorial model of \cite{CKSTT}. Since $N=5$, each generation contains $16$ points; we have explicitly written the multiplicity of each point when it is not one. In zero there are: $0$ points of the first generation, $8$ points of the second, $12$ of the third, $14$ of the fourth and $15$ points of the fifth generation. }}\label{fig2}
\end{minipage}
\end{figure}
Then, once the combinatorial model is fixed, the choice of the embedding in $\R^2$ is equivalent to the choice of the following free parameters:
\begin{itemize}
\item the placement of the first generation $\SS_1$ (which implies the choice of a parameter in $\mathbb R^{2^N}$);
\item the choice of a procreation angle $\vartheta^{\FF}$ for each family of the generation set (which globally implies the choice of a parameter in $\mathbb T^{(N-1)2^{N-2}}$, since $(N-1)2^{N-2}$ is the number of families).
\end{itemize}
We  denote the corresponding generation set by $\SS(\SS_1,\vartheta^{\FF})$ and the space of parameters by $\XX:=\mathbb R^{2^N}\times\mathbb T^{(N-1)2^{N-2}}$.

In Section \ref{genset} we will prove that the set of parameters producing {\em degenerate} generation sets is contained in a  closed set of null measure in $\XX$.

We claim that the set of $(\SS_1,\vartheta^{\FF})\in\XX$ such that $\SS(\SS_1,\vartheta^{\FF})\subset\Q^2\setminus\{0\}$ is dense in $\XX$. This is a consequence of two facts:
\begin{itemize}
\item the density of $\Q^2\setminus\{0\}$ in $\R^2$ (for the placement of the first generation);
\item the density of (non-zero) rational points on circles having a diameter with rational endpoints.
\end{itemize}

These two remarks imply that the set of $(\SS_1,\vartheta^{\FF})\in\XX$ such that $\SS(\SS_1,\vartheta^{\FF})$ is non-degenerate {\em and} $\SS(\SS_1,\vartheta^{\FF})\subset\Q^2\setminus\{0\}$ is dense in $\XX$.

In order to prove the growth of Sobolev norms, we  require a further property on the generation set $\SS$, i.e. the {\em norm explosion property}
\begin{equation}\label{norm.expl}
\sum_{k\in\SS_{N-2}}|k|^{2s}>\frac{1}{2}2^{(s-1)(N-5)}\sum_{k\in\SS_3}|k|^{2s}\ .
\end{equation}
Given $N\gg1$, our aim is to prove the existence of a non degenerate generation set $\SS\subset\Q^2\setminus\{0\}$ satisfying \eqref{norm.expl}. The fact that \eqref{norm.expl} is an open condition on the space of parameters $\XX$, together with the above remarks, implies that it is enough to prove the existence of a  (possibly degenerate) generation set $\SS\subset\R^2$ satisfying \eqref{norm.expl}, which is achieved by the {\em prototype embedding} described in \cite{CKSTT}, pp. 101--102 (see Figure \ref{fig2}).

\begin{remark}\label{Nn}
Note that, for any given positive integer $\ell$, the function $F:\mathbb S^{\ell-1}\to\mathbb R$, where
$$
\mathbb S^{\ell-1}=\left\{(x_1,\ldots,x_\ell)\in\mathbb R^\ell\middle|\sum_{i=1}^\ell x_i^2=1)\right\}\ ,
$$
defined by
$$
F(x_1,\ldots,x_\ell)=\sum_{i=1}^\ell x_i^{2s}
$$
attains its minimum (since $s>1$) at $$(x_1,\ldots,x_\ell)=(\ell^{-1/2},\ldots,\ell^{-1/2})$$ and its maximum at $$(x_1,x_2,\ldots,x_\ell)=(1,0,\ldots,0)\ .$$

From this one deduces that for each family $\FF$ with parents $v_1,v_2$ and children $v_3,v_4$ one must have
$$
\frac{|v_3|^{2s}+|v_4|^{2s}}{|v_1|^{2s}+|v_2|^{2s}}\leq 2^{s-1}
$$
and therefore, for all $1\leq i\leq N-1$,
$$
\frac{\sum_{k\in\SS_{i+1}}|k|^{2s}}{\sum_{k\in\SS_i}|k|^{2s}}\leq 2^{s-1}\ .
$$
which implies
$$
\frac{\sum_{k\in\SS_{j}}|k|^{2s}}{\sum_{k\in\SS_i}|k|^{2s}}\leq 2^{(s-1)(j-i)}\ .
$$
for all $1\leq i\leq j\leq N$. This means that we have to choose $N$ large if we want the ratio
$$
\frac{\sum_{k\in\SS_{N-2}}|k|^{2s}}{\sum_{k\in\SS_3}|k|^{2s}}
$$
to be large.

Moreover, since
$$
F(\ell^{-1/2},\ldots,\ell^{-1/2})=\ell^{-s+1}\,,\qquad F(1,0,\ldots,0)=1
$$
we have for all $1\leq i,j\leq N$
$$
\frac{\sum_{k\in\SS_{j}}|k|^{2s}}{\sum_{k\in\SS_i}|k|^{2s}}\leq n^{s-1}\ .
$$
which implies that also $n$ has to be chosen large enough.

In this sense the prototype embedding and the choice $n=2^{N-1}$ are optimal, because they attain the maximum possible growth of the quantity $\sum_{k\in\SS_i}|k|^{2s}$ both at each step and between the first and the last generation.
\end{remark}

Then, once we are given a non-degenerate generation set contained in $\Q^2\setminus\{0\}$ and satisfying \eqref{norm.expl}, it is enough to multiply by any integer multiple of the least common denominator of its elements in order to get a non-degenerate generation set $\SS\in\Z^2\setminus\{0\}$ and satisfying \eqref{norm.expl} (note that \eqref{norm.expl} is invariant by dilations of the set $\SS$). Note that we can dilate $\SS$ as much as we wish, so we can make $\min_{k\in\SS}|k|$ as large as desired.

These considerations are summarized by the following proposition (the analogue of Proposition 2.1 in \cite{CKSTT}).
\begin{proposition}\label{prop.gen.set}
For all $K,\delta,\RR>0$, there exist $N\gg1$ and a non-degenerate generation set $\SS\subset\Z^2$ such that
\begin{equation}
\frac{\sum_{k\in\SS_{N-2}}|k|^{2s}}{\sum_{k\in\SS_3}|k|^{2s}}\gtrsim\frac{K^2}{\delta^2}
\end{equation}
and such that
\begin{equation}
\min_{k\in\SS}|k|\geq\RR\ .
\end{equation}
\end{proposition}

\subsection{Genericity of non-degenerate generation sets}
\label{genset}

The main result of this section is the following.
\begin{proposition}\label{prop.nondeg}
There exists a closed set of zero measure $\mathcal D\subset\XX$ such that the generation set $\SS(\SS_1,\vartheta^{\FF})$ is non-degenerate for all $(\SS_1,\vartheta^{\FF})\in\XX\setminus\mathcal D$.
\end{proposition}

First, we need a lemma ensuring that any linear relation among the elements of the generation set that is not a linear combination of the family relations is generically not fulfilled.

\begin{lemma}\label{lemma.identities}
Let $\mu\in\mathbb Z^{N2^{N-1}}$, $i=1,\ldots,M$ be an integer vector, linearly independent from the subspace of $\mathbb R^{N2^{N-1}}$ generated by all the vectors $\lambda^{\FF}$ associated to the families. Then,  there for an open set of full measure $\mathfrak S\subset\XX$, one has that if $(\SS_1,\vartheta^{\FF})\in\mathfrak S$, then $\SS(\SS_1,\vartheta^{\FF})$ is such that
\begin{equation}\label{sigma}
\sum_{j=1}^{N2^{N-1}}\mu_{j}v_j\neq 0\ .
\end{equation}
\end{lemma}
\begin{proof}
We denote the elements of $\SS$ by $v_1,\ldots,v_{|\SS|}$, with $|\SS|=N2^{N-1}$. For simplicity and without loss of generality, we order the $v_j$'s so that couples of siblings always have consecutive subindices.


For each family $\FF$, both the linear and the quadratic relations
$$\sum_j \lambda_j^{\FF} v_j=0\,,\quad \sum_j \lambda_j^{\FF} |v_j|^2=0$$
are satisfied. The coefficients of the linear relations can be collected in a matrix $\Lambda^{\FF}$ with $(N-1)2^{N-2}$ rows (as many as the number of families) and $N2^{N-1}$ columns (as many as the elements of $\SS$), so that the linear relations become $$\Lambda^{\FF}v=0\ .$$ We choose to order the rows of $\Lambda^{\FF}$ so that the matrix is in {\em lower row echelon} form (see figure).

$$\begin{array}{c  c c c c c c c c c c c } w_1& w_2& w_3& w_4& w_5& p_5& w_6 & p_6 & w_7 & p_7 & w_8& p_8\\ \\ 1 &1&0&0&-1&-1&\!\! \!\!\!\vline \quad0&0&0&0  &0&0 \\ \cline{7-8}
0&0&1&1&0&0&-1&-1 &\!\! \!\!\!\vline \quad 0&0  &0&0\\ \cline{9-10} 0&0&0&0&1&0&1&0&-1&-1& \!\! \!\!\!\vline \quad 0&0 \\ \cline{11-12}0&0&0&0&0&1&0&1&0&0  &-1&-1 \end{array}
$$

Each row of a matrix in lower row echelon form has a {\em pivot}, i.e. the first nonzero coefficient of the row starting from the right. Being in lower row echelon form means that the pivot of a row is always strictly to the right of the pivot of the row above it. In the matrix $\Lambda^{\FF}$, the pivots are all equal to $-1$ and they correspond to one and only one of the child from each family. In order to use this fact, we accordingly rename the elements of the generation set by writing $v=(p,w)\in\R^{2a}\times\R^{2b}$, with $a=(N-1)2^{N-2}$, $b=N2^{N-1}-a=(N+1)2^{N-2}$, where the $p_j\in\R^2$ are the elements of the generation set corresponding to the pivots and the $w_{\ell}\in\R^2$ are all the others, i.e. all the elements of the first generation and one and only one child (the non-pivot) from each family. Here, the index $\ell$ ranges from $1$ to $b$, while the index $j$ ranges from $2^{N-1}+1$ to $b$ (note that $a+2^{N-1}=b$), so that a couple $(p_j,w_{\ell})$ corresponds to a couple of siblings if and only if $j=\ell$. Then, the linear relations $\Lambda^{\FF}v=0$ can be used to write each $p_j$ as a linear combination of the $w_\ell$'s with $\ell\leq j$ only:
\begin{equation}\label{pivot}
p_j=\sum_{\ell\leq j}\eta_{\ell}w_{\ell}, \quad \eta_{\ell}\in\Q\ .
\end{equation}
Finally, the quadratic relations $\Lambda^{\FF}|v|^2=0$ constrain each $w_{\ell}$ with $\ell>2^{N-1}$ (i.e. not in the first generation) to a circle depending on the $w_j$ with $j<\ell$; note that this circle has positive radius provided that the parents of $w_{\ell}$ are distinct.
Then, eq. \eqref{pivot} implies that the l.h.s. of \eqref{sigma} can be rewritten in a unique way as a linear combination of the $w_{\ell}$'s only, so we have
\begin{equation}\label{sigma2}
\sum_{\ell=1}^b\nu_\ell w_\ell=0\ .
\end{equation}
Hence, the assumption that $\mu$ is linearly independent from the space generated by the $\lambda^{\FF}$'s is equivalent to the fact that $\nu\in\R^{2b}$ does not vanish.

Now, let $$\bar\ell:=\max\left\{\ell\ \middle|\ \nu_\ell\neq0\right\}$$
so that \eqref{sigma2} is equivalent to
\begin{equation}\label{sigma3}
w_{\bar\ell}=-\frac{1}{\nu_{\bar\ell}}\sum_{\ell<\bar\ell}\nu_\ell w_\ell
\end{equation}
If $\bar\ell\leq 2^{N-1}$, then $w_{\bar\ell}$ is in the first generation. Since there are no restrictions (either linear or quadratic) on the first generation, the statement is trivial. Hence assume $\bar\ell>2^{N-1}$. We can assume (by removing from $\XX$ a closed subset of zero measure) that $v_h\neq v_k$ for all $h\neq k$. Then the quadratic constraint on $w_{\bar\ell}\in\R^2$ gives a circle of positive radius. By excluding at most one point this circle, we can ensure that the relation \eqref{sigma3} is not fulfilled, which proves the thesis of the lemma.

\end{proof}

In view of Lemma \ref{lemma.identities}, those vectors $\mu\in\mathbb Z^{N2^{N-1}}$ that are linear combinations of the family vectors assume a special importance, since that is the only case in which the relation $\sum\mu_i v_i=0$ cannot be excluded when constructing the set $\SS$. In that case, we will refer to $\mu\sim0$ as a \textit{formal identity}. In general, we will write $\mu\sim\lambda$ whenever the vector $\mu-\lambda$ is a linear combination of the family relations.

We introduce some more notation: given a vector $\lambda\in\mathbb Z^{N2^{N-1}}$, we denote by $\pi_j\lambda$ the projection of $\lambda$ on the space generated by the members of the $j$-th generation. Now, let $R_{\alpha}=\sum_i\alpha_i\lambda^{\FF_i}$ be a linear combination with integer coefficients of the family vectors. We denote by $n_{R_{\alpha}}$ the number of families on which the linear combination is supported, i.e. the cardinality of $\{i|\alpha_i\neq0\}$. Moreover, we denote by $n_{R_{\alpha}}^k$ the number of families of age $k$ on which $R_{\alpha}$ is supported, i.e. the cardinality of $$\{i|\alpha_i\neq0\ \text{and}\ \FF_i\ \text{is a family of age}\ k\}\ .$$ Finally, we denote respectively by $m_{R_{\alpha}}$ and $M_{R_{\alpha}}$ the minimal and the maximal age of families on which $R_{\alpha}$ is supported. Then, we make the two following simple remarks.

\begin{remark}\label{remark.1}
If $n_{R_{\alpha}}^k=n_{R_{\alpha}}^{k+1}=1$, then $\pi_{k+1}R_{\alpha}$ is supported on at least two distinct elements.
\end{remark}

\begin{remark}\label{remark.2}
If $n_{R_{\alpha}}^k\neq n_{R_{\alpha}}^{k+1}$, then $\pi_{k+1}R_{\alpha}$ is supported on at least two distinct elements.
\end{remark}

Before proving the main result of this section, we need some lemmas.

\begin{lemma}\label{lemma.3.8}
If $n_{R_{\alpha}}\geq3$, then $R_{\alpha}$ is supported on at least 8 distinct elements.
\end{lemma}

\begin{proof}
For simplicity of notation, here we put $m:=m_{R_{\alpha}}$ and $M:=M_{R_{\alpha}}$. First, observe that $\pi_{m}R_{\alpha}$ is supported on $2n_{R_{\alpha}}^m$ elements and that $\pi_{M+1}R_{\alpha}$ is supported on $2n_{R_{\alpha}}^M$ elements. So, if $n_{R_{\alpha}}^m+n_{R_{\alpha}}^M\geq4$, the thesis is trivial.

Up to symmetry between parents and children, we may choose $n_{R_{\alpha}}^m\leq n_{R_{\alpha}}^M$. So, the only non-trivial cases to consider are $(n_{R_{\alpha}}^m,n_{R_{\alpha}}^M)=(1,1)$ and $(n_{R_{\alpha}}^m,n_{R_{\alpha}}^M)=(1,2)$.

\textit{Case (1,1)}. We must have $M\geq m+2$, since there must be at least three families in $R_{\alpha}$. Now, let $C:=\max_i n_{R_{\alpha}}^i$. If $C=1$, then by Remark \ref{remark.1} the support of $R_{\alpha}$ involves at least 4 generations and at least 2 elements for each generation, so it includes at least 8 elements. If $C>1$, then there exist $m\leq i,j<M$ with $i\neq j$ such that $n_{R_{\alpha}}^i<n_{R_{\alpha}}^{i+1}$ and $n_{R_{\alpha}}^{j+1}<n_{R_{\alpha}}^j$. Then, by Remark \ref{remark.2}, $\pi_{i+1}R_{\alpha}$ and $\pi_{j+1}R_{\alpha}$ are supported on at least 2 elements each. Since $\pi_{m}R_{\alpha}$ and $\pi_{M+1}R_{\alpha}$ are supported on exactly 2 elements and since the four indices $m,i+1,j+1,M+1$ are all distinct, then we have the thesis.

\textit{Case (1,2)}. Here, $\pi_{m}R_{\alpha}$ is supported on 2 elements and $\pi_{M}R_{\alpha}$ is supported on 4 elements. Moreover, there exists $m\leq i<M$ such that $n_{R_{\alpha}}^{i+1}<n_{R_{\alpha}}^i$, which by Remark \ref{remark.2} gives us at least 2 elements in the support of $\pi_{i+1}R_{\alpha}$. Thus, we have the thesis.
\end{proof}

From Lemma \ref{lemma.3.8}, the next corollary follows immediately.

\begin{corollary}\label{pochi}
If $R_\alpha$ is supported on at most 7 elements, then $R_\alpha$ is an integer multiple of either a family vector or a resonant vector of type CF.
\end{corollary}

\begin{lemma}\label{lemma.const}
Let $A,B,C\in\mathbb R$, $R>0$ and $p,q\in\mathbb R^2\simeq\mathbb C$ be fixed. Let
$$c_1(\vartheta):=p+Re^{i\vartheta}\ c_2(\vartheta):=p-Re^{i\vartheta}\ .$$
Then, the function $F:\mathbb S^1\to\mathbb R$ defined by
$$F(\vartheta):=A|c_1(\vartheta)|^2+B|c_2(\vartheta)|^2+C-|Ac_1(\vartheta)+Bc_2(\vartheta)+q|^2$$ is an analytic function of $\vartheta$, and it is a constant function only if $A=B$ or if $(A+B-1)p+q=0$.
\end{lemma}

\begin{proof}
An explicit computation yields
$$F(\vartheta)=2R(B-A)\langle(A+B-1)p+q,e^{i\vartheta}\rangle+K$$
where $K$ is a suitable constant that does not depend on $\vartheta$.
\end{proof}

\begin{corollary}\label{coro.isolated}
If $A\neq B$ and $(A+B-1)p+q\neq0$, then the zeros of $F$ are isolated.
\end{corollary}

\begin{lemma}\label{lemma.boh}
Let $\FF$ be a family of age $i$ in $\SS$ and let $\lambda^{\FF_p}$ be the vector corresponding to the sum of the parents of the family $\FF$. Moreover, let $\mu\in\mathbb Z^{|\SS|}$ be another vector with $|\mu|\leq5$, such that $\pi_j\mu=0$ for all $j>i+1$ and such that the support of $\mu$ and the support of the vector $\lambda^{\FF_c}$ corresponding to the children of $\FF$ are disjoint. Finally, let $h,k\in\mathbb Z\setminus\{0\}$. Assume that the formal identity $h\mu+k\lambda^{\FF_p}\sim0$ holds. Then, only two possibilities are allowed:
\begin{enumerate}
\item $h\mu+k\lambda^{\FF_p}=0$;
\item $h\mu+k\lambda^{\FF_p}$ is an integer multiple of $\lambda^{\tilde{\FF}}$, where $\tilde{\FF}$ is a family of age $i-1$, one of whose children is a parent in $\FF$.
\end{enumerate}
\end{lemma}

\begin{proof}
We first remark that $h\mu+k\lambda^{\FF_p}$ is supported on at most 7 elements. Moreover, since it is a linear combination of some family vectors (because of the formal identity $h\mu+k\lambda^{\FF_p}\sim0$), we are in a position to apply Corollary \ref{pochi} and conclude that $h\mu+k\lambda^{\FF_p}$ must be an integer multiple of either a family vector or a resonant vector of type CF.

Now, assume by contradiction that $h\mu+k\lambda^{\FF_p}$ is a nonzero integer multiple of a resonant vector of type CF.  Then, the support of $h\mu+k\lambda^{\FF_p}$ cannot include both parents of the family $\FF$, since the support of a CF vector including a couple of parents of age $i$ should include also a couple of children of age $i+2$, but we know by the assumptions of this lemma that the support of $h\mu+k\lambda^{\FF_p}$ does not include elements of age greater than $i+1$. Therefore, at least one of the elements in $\lambda$ must cancel out with one of the elements in $\lambda^{\FF_p}$, but then the support of $h\mu+k\lambda^{\FF_p}$ can include at most 5 elements, and therefore it cannot be a vector of type CF.

Then, if $h\mu+k\lambda^{\FF_p}$ is a nonzero integer multiple of a single family vector $\tilde{\FF}$, observe that its support must contain one and only one of the parents of $\FF$. In fact, if both canceled out, then the support of $h\mu+k\lambda^{\FF_p}$ could contain at most 3 elements, which is absurd. If none of them canceled out, then we should have $\tilde{\FF}=\FF$, which in turn is absurd since, by the assumptions of this lemma, the support of $h\mu+k\lambda^{\FF_p}$ cannot include any of the children of the family $\FF$. This concludes the proof of the lemma.
\end{proof}

We can now prove the main proposition.

\begin{proof}[Proof of Proposition \ref{prop.nondeg}]
The proof is based on the following induction procedure. At each step, we assume to have already fixed $i$ generations and say $h<2^{N-2}$ families with children in the $i+1$-th generation. Our induction hypothesis is that the non-degeneracy condition is satisfied for the vectors $\mu$ whose support involves only the elements that we have already fixed. Then, our aim is to show that the non-degeneracy condition holds true also for the set for the vectors supported on the already fixed elements plus the two children of a new family (whose procreation angle has to be accordingly chosen) with children in the $i+1$-th generation, up to removing from $\XX$ a closed set of null measure.

First, we observe that, at the inductive step zero, i.e. when placing the first generation $\SS_1$, the set of parameters that satisfy both non-degeneracy and non-vanishing of any fixed finite number of linear relations that are not formal identities is obviously open and of full measure.

Then, we have to study what happens when choosing a procreation angle, i.e. when generating the children of a family $\FF=(p_1,p_2;c_1,c_2)$ whose parents have already been fixed. Let $\lambda_1,\lambda_2\in\mathbb Z^{|\SS|}$ be defined by $\sum_j(\lambda_i)_jv_j=c_i$ for $i=1,2$. We need to study the non-degeneracy condition associated to the vector $\lambda(A,B,\mu)\in\mathbb Z^{|\SS|}$ given by
$$\lambda(A,B,\mu):=A\lambda_1+B\lambda_2+\mu\ ,$$
where $\mu$ satisfies the same properties as in the assumptions of Lemma \ref{lemma.boh} and
$$|A|+|B|+|\mu|\leq5$$
$$A+B+\sum_j\mu_j=1\ .$$
If $A\neq B$ and if $(A+B-1)(p_1+p_2)+2\sum_j\mu_jv_j\neq0$, then we are done, because, thanks to Corollary \ref{coro.isolated}, the non-degeneracy condition is satisfied for any choice of the generation angle except at most a finite number. Therefore, we have to study separately the case $A=B$ and, for $A\neq B$, we have to prove that $(A+B-1)\lambda^{\FF_p}+2\mu\sim0$ holds as a formal identity only in the cases allowed by Definition \ref{def.ND}. Whenever the formal identity $(A+B-1)\lambda^{\FF_p}+2\mu\sim0$ does not hold, we can impose $(A+B-1)(p_1+p_2)+2\sum_j\mu_jv_j\neq0$ by just removing from $\XX$ a closed set of measure zero, thanks to Lemma \ref{lemma.identities}.

\textit{Case $A=B$}. If $(A,B)=(0,0)$ there is nothing to prove, thanks to the induction hypothesis. Then we have to study $(A,B)=\pm(1,1)$. In this case, thanks to the linear relation defining the family $\FF$, we have the formal identity
$$\lambda(A,B,\mu)\sim\pm\lambda^{\FF_p}+\mu=:\nu^{\pm}$$
with $|\mu|\leq3$, $\sum_j\nu_j^{+}=-1$ and $\sum_j\nu_j^{-}=3$. The good point is that $\nu^{\pm}$ is entirely supported on the elements of the generation set that have already been fixed, so we can apply the induction hypothesis of non-degeneracy to $\nu^{\pm}$ and distinguish the 4 cases given by Definition \ref{def.ND}: we have to verify that $\lambda(A,B,\mu)$ accordingly falls into one of the allowed cases.
\begin{itemize}
\item $\nu^{\pm}$ satisfies case 1 of Definition \ref{def.ND}. Then one readily verifies that $\lambda(A,B,\mu)$ satisfies either case 2 or case 3 of Definition \ref{def.ND}.
\item $\nu^{\pm}$ satisfies case 2 of Definition \ref{def.ND}. Observe that the family involved by the statement of case 2 cannot be $\FF$, since $\nu^{\pm}$ cannot be supported on either child of the family $\FF$. Then $\mu$ must cancel out one of the two parents appearing in $\pm\lambda^{\FF_p}$. It cannot be supported on both parents because that would not be consistent with $|\mu|\leq3$ and $|\nu^{\pm}|=3$. Then one verifies that $\lambda(A,B,\mu)$ satisfies case 4 of Definition \ref{def.ND}.
\item $\nu^{\pm}$ satisfies case 3 of Definition \ref{def.ND}. Since $|\nu^{\pm}|=1$, then nothing cancels out, so the support of $\nu^{\pm}$ includes both parents of $\FF$. But this is absurd, so this case cannot happen.
\item $\nu^{\pm}$ satisfies case 4 of Definition \ref{def.ND}. This case is again absurd, since $\nu^{\pm}$ should be supported on 5 of the 6 elements of a CF vector, including the two parents of the family $\FF$.
\end{itemize}

\textit{Case $A\neq B$}. By symmetry, we may impose $|A|>|B|$. Assume that $(A+B-1)\lambda^{\FF_p}+2\mu\sim0$ holds as a formal identity: we must prove that this can be true only in the cases allowed by Definition \ref{def.ND}. First, we consider the case $A+B-1=0$: then, we must have the formal identity $\mu\sim0$ with $|\mu|\leq5$: so, by Corollary \ref{pochi} either $\mu$ is (up to the sign) a family vector (which may happen only if $(A,B)=(1,0)$ due to the constraint $|A|+|B|+|\mu|\leq5$) or $\mu=0$. Consider the case $(A,B)=(1,0)$: if $\mu$ is a family vector, then $\lambda(A,B,\mu)$ falls into case 3 of Definition \ref{def.ND}; if $\mu=0$, then $\lambda(A,B,\mu)$ falls into case 1 of Definition \ref{def.ND}. If $(A,B)=(2,-1)$ or $(A,B)=(3,-2)$, then $\mu=0$. Then, in both cases, from $$\sum_j\lambda_j(A,B,\mu)|v_j|^2-\bigg|\sum_j\lambda_j(A,B,\mu)v_j\bigg|^2$$
with some explicit computations one deduces $|c_1-c_2|^2=0$ which is absurd, since the induction hypothesis implies $p_1\neq p_2$ and since the endpoints of a diameter of a circle with positive radius are distinct.

Now, if $A+B-1\neq0$ we can apply Lemma \ref{lemma.boh} and deduce that $(A+B-1)\lambda^{\FF_p}+2\mu$ is either zero or an integer multiple of the vector of a family where one of the parents of $\FF$ appears as a child. Suppose first $(A+B-1)\lambda^{\FF_p}+2\mu=0$. Then $A+B-1$ must be even. If $(A,B)=(-1,0)$, then $\mu=\lambda^{\FF_p}$ and $\lambda(A,B,\mu)$ falls into case 2 of Definition \ref{def.ND}. If $(A,B)=(2,1)$, then $\mu=-\lambda^{\FF_p}$ and $\lambda(A,B,\mu)$ falls into case 3 of Definition \ref{def.ND}. These are the only possible cases if $(A+B-1)\lambda^{\FF_p}+2\mu=0$. Finally, assume that $(A+B-1)\lambda^{\FF_p}+2\mu$ is an integer multiple of the vector of a family where one of the parents of $\FF$ appears as a child. Then $\mu$ must be such that the other parent of $\FF$ is canceled out, so $A+B-1$ again has to be even. If $(A,B)=(-1,0)$, then $\mu-\lambda^{\FF_p}$ is the vector of a family where one of the parents of $\FF$ appears as a child and $\lambda(A,B,\mu)$ falls into case 4 of Definition \ref{def.ND}. This is the only possible case, since the support of $\mu$ must include one parent of $\FF$ and the other three members of the family where the other parent of $\FF$ appears as a child. This also concludes the proof of the proposition.
\end{proof}

\section{Proof of Theorem \ref{main.thm}}\label{all.together.now}
In the previous sections we have proved the existence of non-degenerate sets $\SS$ on which the Hamiltonian is \eqref{hamish} and the existence of a  {\em slider solution} for its dynamics. We now turn to the NLS equation \ref{NLS} with the purpose of proving the persistence of this type of solution.

As in \cite{CKSTT}, one can easily prove that \eqref{NLS} is locally well-posed in $\ell^1(\Z^2)$: to this end, one first introduces the multilinear operator
$$\mathcal N(t):\ell^1(\Z^2)\times\ell^1(\Z^2)\times\ell^1(\Z^2)\times\ell^1(\Z^2)\times\ell^1(\Z^2)\to\ell^1(\Z^2)$$
defined by
\begin{equation}\label{defN}
\big(\mathcal N(t)(a,b,c,d,f)\big)_j:=\sum_{j_1,j_2,j_3,j_4,j_5\in\mathbb{Z}^2 \atop j_1+j_2+j_3-j_4-j_5=j}a_{j_1}b_{j_2}c_{j_3}\bar{d}_{j_4}\bar{f}_{j_5}e^{i\omega_6t}
\end{equation}
so that \eqref{NLS} can be rewritten as
$$-i\dot a_j=\big(\mathcal N(t)(a,a,a,a,a)\big)_j\ .$$
Then, in order to obtain local well-posedness it is enough to observe that the following multilinear estimate holds
\begin{equation}\label{stimamultilin}
\|\mathcal N(t)(a,b,c,d,f)\|_{\ell^1}\lesssim\|a\|_{\ell^1}\|b\|_{\ell^1}\|c\|_{\ell^1}\|d\|_{\ell^1}\|f\|_{\ell^1}\ .
\end{equation}
\begin{lemma}\label{lemma.approx1}
Let $0 < \sigma < 1$ be an absolute constant (all implicit constants in this lemma may depend on $\sigma$). Let $B \gg 1$, and let $T \ll B^4\log B$ . Let $g(t) :=\{g_j(t)\}_{j \in \Z^2}$  be a solution of the equation
\begin{equation}\label{tao1}
\dot g(t)= i \big (\mathcal N(t)(g(t),g(t),g(t),g(t),g(t))+\mathcal E(t) \big )
\end{equation}

for times $0 \leq  t \leq  T$ , where $\mathcal N (t )$ is defined in \eqref{defN} and  the initial data $g(0)$  is compactly supported. Assume also that the solution $g(t)$ and the {\em error term} $\mathcal E(t)$ obey the bounds of the form
\begin{equation}\label{tao2}
\| g(t)\|_{\ell^1(\Z^2)}\lesssim B^{-1}
\end{equation}
\begin{equation}\label{tao3}
\left\| \int_0^t \mathcal E(\tau) d\tau \ \right\|_{\ell^1(\Z^2)} \lesssim B^{-1} 
\end{equation}
We conclude that if $a(t)$  denotes the solution to the NLS  \eqref{NLS} with initial data $a(0)= g(0)$  then we have
\begin{equation}
\| a(t)-g(t)\|_{\ell^1(\Z^2)}\lesssim B^{-1-\sigma/2}
\end{equation}
for all $0\leq t \leq T$.
\end{lemma}

\begin{proof}
The proof is the transposition to the quintic case of the proof of Lemma 2.3 of \cite{CKSTT} and is postponed to Appendix \ref{lemmasublime}.
\end{proof}

Given $\delta$, $K$, construct $\SS$ as in Proposition \ref{prop.gen.set}. Note that we are free to specify $\mathcal R=\mathcal R(\delta,K)$  (which measures the inner radius of the frequencies involved in $\SS$) as large as we wish. This construction fixes $N= N(\delta,K)$ (the number of generations). We introduce a further parameter $\epsilon$ (which we are free to specify as a function of $\delta,K$)
and construct the slider  solution $b(t)$ to the toy model concentrated at scale $\epsilon$ according to Proposition \ref{prop.slider} above. This proposition also gives us a time $T_0=T_0(K,\delta)$. Note that the toy model has the following scaling
$$
b^{(\lambda)}(t):=\lambda^{-1}b\left(\frac{t}{\lambda^4}\right).
$$
We choose the initial data for NLS by setting
\begin{equation}\label{piffero}
a_j(0)=b_i^{(\lambda)}(0)\,,\qquad  \text{ for all} \quad j\in \SS_i
\end{equation}
and $a_j(t)=0$ when $j\notin \SS$. 
We want to apply the Approximation Lemma \ref{lemma.approx1} with a parameter $B$ chosen large enough so that
\begin{equation}\label{eq.b.lambda}
B^4\log B\gg\lambda^4T_0.
\end{equation}
We set $g(t)=\{g_j(t)\}_{j\in \Z^2}$ defined by the slider solution as
$$
g_j(t)= b_i^{(\lambda)}(t)\; \forall\; j\in \SS_i
$$
$g_j(t)=0$ otherwise.
Then  we set $\mathcal E(t):= \{\mathcal E_j(t)\}_{j\in \Z^2}$ with
$$
\mathcal E_j(t)= -\sum_{ k_i \in \SS: k_1+k_2+k_3-k_4-k_5=j \atop \omega_6\neq 0} g_{k_1}g_{k_2}g_{k_3}\bar g_{k_4}\bar g_{k_5} e^{i\omega_6 t}
$$
where $\omega_6 = |k_1|^2+|k_2|^2+|k_3|^2-|k_4|^2-|k_5|^2-|j|^2$. We recall that the frequency support of $g(t)$ is in $\SS$ for all times. We choose $B=C(N)\lambda$ and then show that for large enough $\lambda$ the required conditions \eqref{tao2}, \eqref{tao3} hold true. Observe that \eqref{eq.b.lambda} holds true with this choice for large enough $\lambda$. Note first that simply by considering its support, the fact that $|\SS|=C(N)$, and the fact that $\|b(t)\|_{\ell^{\infty}}\sim1$, we can be sure that $\|b(t)\|_{\ell^1(\Z)}\sim C(N)$ and therefore
\begin{equation}\label{falala}
\|b^{(\lambda)}(t)\|_{\ell^1(\Z)},\ \|g(t)\|_{\ell^1(\Z^2)}\leq\lambda^{-1}C(N).
\end{equation}
Thus, \eqref{tao2} holds with the choice $B=C(N)\lambda$. For the second condition \eqref{tao3}, we claim
\begin{equation}\label{claim.tao3}
\left\|\int_0^t\mathcal E(\tau) d\tau\right\|_{\ell^1}\lesssim C(N)(\lambda^{-5}+\lambda^{-9}T).
\end{equation}
This implies \eqref{tao3} since  $B = \lambda C(N)$ and $T = \lambda^4 T_0$. 

We now prove \eqref{claim.tao3}. Since $\omega_6$ does not vanish in the sum defining $\mathcal E$, we can replace $e^{i \omega_6 \tau}$  by $\frac{d}{d \tau}\left[\frac{e^{i \omega_6 \tau}}{i\omega_6}\right]$ and then integrate by parts. Three terms arise: the boundary terms at $\tau= 0, T$ and the integral term involving 
$$\frac{d}{d \tau}\left[g_{k_1}g_{k_2}g_{k_3}\bar g_{k_4}\bar g_{k_5}\right]. $$ For the boundary terms, we
use \eqref{falala} to obtain an upper bound of $C(N)\lambda^{-5}$ . For the integral term, the
$\tau$ derivative falls on one of the $g$ factors. We replace this differentiated term using the equation to get an expression that is 9-linear in $g$ and bounded by $C(N)\lambda^{-9}T$.
Once $\lambda$ has been chosen as above, we choose $\mathcal R$ sufficiently large so that the initial data $g(0) = a(0)$ has the right size:
\begin{equation}\label{tagliadelta}
\left(\sum_{j\in \SS}|g_j(0)|^2|j|^{2s}\right)^{\frac12}\sim \delta.
\end{equation}
This is possible since the quantity on the left scales like $\lambda^{-1}$  and $\mathcal R^s$ respectively  in the parameters $\lambda,\mathcal R$. The issue here is that our choice of frequencies $\SS$ only
gives us a large factor (that is $\frac K\delta$) by which the Sobolev norm of the solution  will grow. If our initial datum is much smaller than $\delta$ in size, the Sobolev norm of the solution will not grow to be larger than $K$.
It remains to show that we can guarantee
\begin{equation}
\left(\sum_{j\in\Z^2}|a_j(\lambda^4T_0)|^2|j|^{2s}\right)^{\frac12}\geq K,
\end{equation}
where $a(t)$ is the evolution of the initial datum $g(0)$ under the NLS.
We do this by first establishing  
\begin{equation}\label{ggrande}
\left(\sum_{j\in\SS}|g_j(\lambda^4T_0)|^2|j|^{2s}\right)^{\frac12}\gtrsim K,
\end{equation}
and then 
\begin{equation}\label{diffe}
\sum_{j\in\SS}|g_j(\lambda^4T_0)-a_j(\lambda^4T_0)|^2|j|^{2s}\lesssim 1.
\end{equation}
In order to prove \eqref{ggrande}, consider the ratio:
\begin{equation}\label{ggrande2}
\mathcal Q:= \frac{\sum_{j\in\SS}|g_j(\lambda^4T_0)|^2|j|^{2s}}{\sum_{j\in\SS}|g_j(0)|^2|j|^{2s}} =
\frac{\sum_{i=1}^N |b_i^{(\lambda)}(\lambda^4 T_0)|^2\sum_{j\in \SS_i}|j|^{2s}}{\sum_{i=1}^N |b_i^{(\lambda)}(0)|^2\sum_{j\in \SS_i}|j|^{2s}}.
\end{equation}
Set $\mathfrak J_i:=\sum_{j\in \SS_i}|j|^{2s}$, by construction $\frac{\mathfrak J_i}{\mathfrak J_{j}} \sim 2^{i-j}$ and by the choice of $N$ one has $\frac{\mathfrak J_{3}}{\mathfrak J_{N-2}}\lesssim \delta^2 K^{-2}$. Then one has
$$
\frac{\sum_{i=1}^N |b_i^{(\lambda)}(\lambda^4 T_0)|^2\mathfrak J_i}{\sum_{i=1}^N |b_i^{(\lambda)}(0)|^2\mathfrak J_i}\gtrsim\frac{\mathfrak J_{N-2}(1-\epsilon)}{\epsilon \sum_{i\neq 3}\mathfrak J_i +(1-\epsilon)\mathfrak J_{3}}=
$$
$$
=\frac{(1-\epsilon)}{\epsilon \sum_{i\neq 3}\frac{\mathfrak J_i}{\mathfrak J_{N-2}} +(1-\epsilon)\frac{\mathfrak J_{3}}{\mathfrak J_{N-2}}}= \frac{1}{\frac{\mathfrak J_{3}}{\mathfrak J_{N-2}}+O(\epsilon)}\gtrsim \frac{K^2}{\delta^2}.
$$
provided that $\epsilon= \epsilon(N,K,\delta)$ is sufficiently small.

In order to prove \eqref{diffe}, we use  the Approximation Lemma \ref{lemma.approx1} we obtain that
\begin{equation}\label{fottiti}
\sum_{j\in \SS} |g_j(\lambda^4 T_0)-a_j(\lambda^4 T_0) |^2|j|^{2s}\lesssim \lambda^{-2-\sigma}  \sum_{j\in \SS}|j|^{2s} \leq \frac12.
\end{equation} 
The last inequality is obtained by scaling $\lambda$ by some (big) parameter $C$ and  $\mathcal R$ by $C^{1/s}$ so that the bound \eqref{tagliadelta} still holds while $\lambda^{-2-\sigma}  \sum_{n\in \SS}|j|^{2s}$ scales as $C^{-\sigma}$.

\begin{appendix}
\section{Proof of Proposition \ref{covering}}\label{slider-sol}
This proof is in fact exactly the same as in \cite{CKSTT}, however in that paper all the results are stated for the cubic case (even though they are clearly more general) and so we give a schematic overview of the main steps.

\begin{lemma}\label{costanti}
Suppose that $[0,\tau]$ is a time interval on which
we have the smallness condition
$$
\int_0^\tau |c(s)|^2ds \lesssim 1
$$
then we have the estimates:
\begin{eqnarray*}
|c_{j\pm 1}^-(\tau)|&\lesssim& e^{-\sqrt{3}\tau}|c_{j\pm 1}^-(0)|+\int_0^\tau e^{-\sqrt{3}(\tau-s)} |c_{j\pm 1}^+(s)||c_{\neq j\pm 1}|^2\,, \\
|c_{j\pm 1}^+(\tau)|&\lesssim& e^{\sqrt{3}\tau}|c_{j\pm 1}^+(0)|+\int_0^\tau e^{\sqrt{3}(\tau-s)} |c_{j\pm 1}^-(s)||c_{\neq j\pm 1}|^2\,,  \\
|c_{j\pm 1}(\tau)|&\lesssim& e^{\sqrt{3}\tau}|c_{j\pm 1}(0)|\,, \\
|c^*(\tau)|&\lesssim& |c^*(0)|\,.
\end{eqnarray*}
\end{lemma}
\begin{proof}
As in \cite{CKSTT} this Lemma follows from equations \eqref{bruttone} by Gronwall's inequality and the definition of $O(\cdot)$.
\end{proof}

We now prove that the incoming target covers the ricochet target. We start from some basic upper bounds on the flow.
\begin{proposition}\label{cazzomerda}
  Let $b(\tau)$ be a solution to toy model such that $b(0)$ is within $(M_j^-,d_j^-,R_j^-)$. Let $c(\tau)$ denote the coordinates of $b(\tau)$ as in \eqref{bruttone} Then, for all $0\le \tau\leq T$ we have the bounds:
\begin{equation}
\begin{array}{rcl} |c^*(\tau)| &= &O(T^{A_j^-}e^{-2\sqrt{3}T}) \\  |c^-_{j-1}(\tau)| &= &O(\sigma e^{-\sqrt{3}\tau}) \\ |c^+_{j-1}(\tau)| &= &O(T^{2A_j^-+1} e^{-4\sqrt{3}T+\sqrt{3}\tau})\\
|c^-_{j+1}(\tau)| &=& O(r_j^-(1+\tau) e^{-2\sqrt{3}T-\sqrt{3}\tau})\\
|c^+_{j+1}(\tau)| &=& O(r_j^- e^{-2\sqrt{3}T+\sqrt{3}\tau})
 \end{array}
 \end{equation}
  \end{proposition}
\begin{proof}
This is Proposition 3.2 of \cite{CKSTT}. The proof is an application of the continuity method and of Lemma \ref{costanti}.
\end{proof}
Now from these basic upper bounds  and from the equations of motion \eqref{bruttone}, \eqref{bruttine} we deduce improved upper bounds on the dynamical variables. We first consider $c_{j-1}^-$,
we have
$$
\dot c_{j-1}^-=-\sqrt{3}c_{j-1}^-+\OO((c_{j-1}^-)^3)+\OO((c_{j-1}^-)^5)+O(T^{A_j^-}e^{-2\sqrt{3}T})
$$
for some explicit expression $\OO((c_{j-1}^-)^3)+\OO((c_{j-1}^-)^5)$.
Now let $g$ be the solution to the corresponding equation
$$
\dot g= -\sqrt{3} g +\OO(g^3)+\OO(g^5)\,,
$$
with the same initial datum $g(0)=\sigma$.
One has the bound
\begin{equation}\label{gbound}
g(\tau)=O(\sigma e^{\sqrt{3}\tau})\,,
\end{equation}
which is formula (3.35) of \cite{CKSTT}. Then by estimating the {\em error term} $E_{j-1}^-:=c_{j-1}^--g$ one has
\begin{subequations}\label{pallegen}
\begin{align}
c_{j-1}^-(\tau) &= g(\tau)+O(T^{A_j^-+1}e^{-2\sqrt{3}T})\,\label{palle1}\\
\OO(c^2)&=\OO(g^2)+O(T^{A_j^-+1}e^{-2\sqrt{3}T})\label{palle2}\\
\OO(c^2_{\neq j+1})&=\OO(g^2)+O(T^{A_j^-+1}e^{-2\sqrt{3}T-\sqrt{3}\tau})\label{palle3}
\end{align}
\end{subequations}
which are respectively formul\ae (3.36)-(3.38) of \cite{CKSTT}.
Now we control the leading peripheral modes. Inserting \eqref{palle2} in \eqref{bruttone5} we see that
$$
\dot c_{\geq j+2}= i \kappa c_{\geq j+2}+ \OO(c_{\geq j+2} g^2)+ \OO(c_{\geq j+2} g^4)+O(T^{A_j^-}e^{-2\sqrt{3}T}|c_{\geq j+2}|)
$$
We approximate this by the corresponding linear equation
$$
\dot u= i \kappa u+ \OO(u g^2)+\OO(u g^4)
$$
where $u(\tau) \in \mathbb C^{N-j-1}$. This equation has a fundamental solution $G_{\geq j+2}(\tau) : \mathbb C^{N-j-1}\to \mathbb C^{N-j-1}$.  From \eqref{gbound} and Gronwall's inequality we have
\begin{equation}\label{gint}
\int_0^Tg^2(\tau)d\tau=O(1)
\end{equation}
and
\begin{equation}\label{funda}
|G_{\geq j+2}|,|G_{\geq j+2}^{-1}|=O(1)\ .
\end{equation}
Setting $c_{\geq j+2}(0)=e^{-2\sqrt{3}T}a_{\geq j+2}+O(T^{A_j^-}e^{-3\sqrt{3}T})$, we define
$$
E_{\geq j+2}:=c_{\geq j+2}-e^{-2\sqrt{3}T}G_{\geq j+2}a_{\geq j+2}\ .
$$
Applying the bound on $c_{\geq j+2}$ from Proposition \ref{cazzomerda} and Gronwall's inequality, we conclude
$$
|E_{\geq j+2}(\tau)|=O(T^{A_j^-}e^{-3\sqrt{3}T})
$$
for all $0\leq\tau\leq T$ and thus
\begin{equation}\label{cj2}
c_{\geq j+2}(\tau)=e^{-2\sqrt{3}T}G_{\geq j+2}(\tau)a_{\geq j+2}+O(T^{A_j^-}e^{-3\sqrt{3}T})\ .
\end{equation}
This is formula (3.41) of \cite{CKSTT}.


Now we consider the two leading secondary modes $c^+_{j+1},c^-_{j+1}$ simultaneously. From \eqref{bruttone}, \eqref{pallegen} and Proposition \ref{cazzomerda}, we have the system
$$
\begin{pmatrix} \dot c^-_{j+1}\\ \dot c^+_{j+1}
\end{pmatrix}=\sqrt{3} \begin{pmatrix} - c^-_{j+1}\\  c^+_{j+1}
\end{pmatrix} + M(\tau) \begin{pmatrix}  c^-_{j+1}\\  c^+_{j+1}
\end{pmatrix} +\begin{pmatrix} O(T^{A_j^-+1}e^{-4\sqrt{3}T})\\  O(T^{A_j^-+1}e^{-4\sqrt{3}T+\sqrt{3}\tau})
\end{pmatrix}
$$
Here $M(\tau)$ is a two by two matrix with entries all $\OO(g^2)+\OO(g^4)$.
Passing to the variables
$$
\tilde a_{j+1}(\tau):=\begin{pmatrix} \tilde a^-_{j+1}(\tau)\\  \tilde a^+_{j+1}(\tau)
\end{pmatrix}
$$
where
$$
\tilde a^-_{j+1}(\tau)=e^{2\sqrt{3}T+\sqrt{3}\tau}c^-_{j+1}(\tau)\,, \quad \tilde a^+_{j+1}(\tau)=e^{2\sqrt{3}T-\sqrt{3}\tau}c^+_{j+1}(\tau)\ ,
$$
we get the equation
\begin{equation}\label{laeva}
\left\{
\begin{array}{rl}
\partial_\tau\tilde a_{j+1}(\tau)&=A(\tau)\tilde a_{j+1}(\tau)+O(T^{A_j^-+1}e^{-2\sqrt{3}T+\sqrt{3}\tau})\\
\tilde a_{j+1}(0)&=a_{j+1}+O(T^{A_j^-}e^{-\sqrt{3}T})
\end{array}
\right.
\end{equation}
where $A(\tau)$ is some known matrix which (by \eqref{gbound}) has bounds
$$
A(\tau)=\sigma^2\begin{pmatrix} O(e^{-2\sqrt{3}\tau}) & O(1) \\ O(e^{-4\sqrt{3}\tau}) & O(e^{-2\sqrt{3}\tau})
\end{pmatrix}\ .
$$
We have obtained formula (3.42) of \cite{CKSTT}. Hence, following verbatim the proof given in \cite{CKSTT}, we get
\begin{equation}\label{cj1}
\begin{pmatrix}
e^{2\sqrt{3}T+\sqrt{3}\tau}c^-_{j+1} \\ e^{2\sqrt{3}T-\sqrt{3}\tau}c^+_{j+1}
\end{pmatrix}
=G_{j+1}(\tau)a_{j+1}+O(T^{A_j^-+2}e^{-\sqrt{3}T})
\end{equation}
which is formula (3.45) of \cite{CKSTT}.

Then, following Section 3.7 of \cite{CKSTT} verbatim, we deduce that the incoming target covers the ricochet target.

\smallskip

Then, one has to prove that the ricochet target covers the outgoing target. In order to do this, one should adapt sections 3.8-3.9 of \cite{CKSTT} exactly as we have done in the previous section. Since this is completely straightforward, we will not write it down.

\smallskip

The last step consists in proving that the outgoing target $(M_j^+,d_j^+,r_j^+)$ covers the next incoming target $(M_{j+1}^-,d_{j+1}^-,r_{j+1}^-)$. An initial datum in the outgoing target has the form
\begin{eqnarray*}
c_{\leq j-1}(0)&=&O(T^{A_j^+}e^{-2\sqrt{3}T})\\
c_{j+1}^-(0)&=&O(T^{A_j^+}e^{-4\sqrt{3}T})\\
c_{j+1}^+(0)&=&\sigma+O(T^{A_j^+}e^{-\sqrt{3}T})\\
c_{\geq j+2}(0)&=&e^{-2\sqrt{3}T}a_{\geq j+2}+O(T^{A_j^+}e^{-3\sqrt{3}T})
\end{eqnarray*}
for some $a_{\geq j+2}$ of magnitude at most $r_j^+$. From \eqref{bruttone5}, \eqref{cacca3}, \eqref{cacca4} we deduce
$$
\dot c_{\neq j+1}=O(|c_{\neq j+1}|)\ .
$$
Thus, for all $0\leq\tau\leq10\log\frac{1}{\sigma}$, Gronwall's inequality gives
\begin{equation}\label{pippaboh}
c_{\neq j+1}(\tau)=O\left(\frac{1}{\sigma^{O(1)}}T^{A_j^+}e^{-2\sqrt{3}T}\right)\ .
\end{equation}
The stable leading mode $c_{j+1}^-$ can be controlled by \eqref{bruttone3}, which by \eqref{pippaboh}
becomes
$$
\dot c_{j+1}^-=O(|c_{j+1}^-|)+O\left(\frac{1}{\sigma^{O(1)}}T^{2A_j^+}e^{-4\sqrt{3}T}\right)\ .
$$
By Gronwall's inequality we conclude
\begin{equation}\label{pippaboh2}
c_{j+1}^-(\tau)=O\left(\frac{1}{\sigma^{O(1)}}T^{2A_j^+}e^{-4\sqrt{3}T}\right)\ .
\end{equation}
Then, taking the $c_{j+1}^+$ component of \eqref{cacca1} we obtain, by \eqref{pippaboh} and \eqref{pippaboh2}
$$
\dot c_{j+1}^+=\sqrt{3}(1-|c_{j+1}^+|^2)c_{j+1}^++O\left(\frac{1}{\sigma^{O(1)}}T^{2A_j^+}e^{-4\sqrt{3}T}\right)\ .
$$
As in \cite{CKSTT}, we define $\hat g$ to be the solution to the ODE
\begin{equation}\label{hatg}
\partial_\tau\hat g=\sqrt{3}(1-|\hat g|^2)\hat g
\end{equation}
with initial datum $\hat g(0)=\sigma$. Such solution can be easily computed and is given by
$$
\hat g(\tau)=\frac{1}{\sqrt{1+e^{-2\sqrt{3}(\tau-\tau_0)}}}
$$
where $\tau_0$ is defined by
$$
\frac{1}{\sqrt{1+e^{2\sqrt{3}\tau_0}}}=\sigma\ .
$$
We note that
$$
\hat g(2\tau_0)=\frac{1}{\sqrt{1+e^{-2\sqrt{3}\tau_0}}}=\sqrt{1-\sigma^2}
$$
and that $2\tau_0\leq10\log\frac{1}{\sigma}$ if $\sigma$ is small enough. Then, estimating as in \cite{CKSTT} (via Gronwall's inequality) the error
$$
E_{j+1}^+:=c_{j+1}^+-\hat g\ ,
$$
we get
\begin{equation}\label{belloschifo}
c_{j+1}^+(\tau)=\hat g(\tau)+O\left(\frac{1}{\sigma^{O(1)}}T^{A_j^+}e^{-\sqrt{3}T}\right)\ .
\end{equation}
This (together with \eqref{pippaboh} and \eqref{pippaboh2}) implies
\begin{equation}\label{belloschifo2}
\OO(c^2)+\OO(c^4)=\OO(\hat g^2)+\OO(\hat g^4)+O\left(\frac{1}{\sigma^{O(1)}}T^{A_j^+}e^{-\sqrt{3}T}\right)\ .
\end{equation}
Now, from \eqref{bruttone3}, \eqref{pippaboh} and \eqref{belloschifo2}, we have
$$
\dot c_{\geq j+2}=i\kappa c_{\geq j+2}+\OO(\hat g^2 c_{\geq j+2})+\OO(\hat g^4 c_{\geq j+2})+O\left(\frac{1}{\sigma^{O(1)}}T^{2A_j^+}e^{-3\sqrt{3}T}\right)\ .
$$
We approximate this flow by the linear equation
$$
\dot u=i\kappa u+\OO(u\hat g^2)+\OO(u\hat g^4)
$$
where $u(\tau)\in\mathbb C^{N-j-1}$. This equation has a fundamental solution $\hat G_{\geq j+2}(\tau):\mathbb C^{N-j-1}\to\mathbb C^{N-j-1}$ for all $\tau\geq0$; from the boundedness of $\hat g$ and Gronwall's inequality we get
\begin{equation}\label{gigi}
|\hat G_{\geq j+2}(\tau)|,|\hat G_{\geq j+2}^{-1}(\tau)|\lesssim\frac{1}{\sigma^{O(1)}}\ .
\end{equation}
As in \cite{CKSTT}, a Gronwall estimate of the error
$$
E_{\geq j+2}(\tau):=c_{\geq j+2}(\tau)-e^{-2\sqrt{3}T}\hat G_{\geq j+2}(\tau)a_{\geq j+2}
$$
gives
\begin{equation}\label{argh}
c_{\geq j+2}(\tau)=e^{-2\sqrt{3}T}\hat G_{\geq j+2}(\tau)a_{\geq j+2}+O\left(\frac{1}{\sigma^{O(1)}}T^{2A_j^+}e^{-3\sqrt{3}T}\right)
\end{equation}
which is equation (3.62) of \cite{CKSTT}. Then, at the time $\tau=2\tau_0\leq10\log\frac{1}{\sigma}$, the estimates become
\begin{eqnarray*}
c_{\leq j-1}(2\tau_0)&=&O\left(\frac{1}{\sigma^{O(1)}}T^{A_j^+}e^{-2\sqrt{3}T}\right)\\
c_{j+1}^-(2\tau_0)&=&O\left(\frac{1}{\sigma^{O(1)}}T^{2A_j^+}e^{-4\sqrt{3}T}\right)\\
c_{j+1}^+(2\tau_0)&=&\sqrt{1-\sigma^2}+O\left(\frac{1}{\sigma^{O(1)}}T^{A_j^+}e^{-\sqrt{3}T}\right)\\
c_{\geq j+2}(2\tau_0)&=&e^{-2\sqrt{3}T}\hat G_{\geq j+2}(2\tau_0)a_{\geq j+2}+O\left(\frac{1}{\sigma^{O(1)}}T^{2A_j^+}e^{-3\sqrt{3}T}\right)\ .
\end{eqnarray*}
From this, we deduce
$$
|b_j|=\left(1-\sum_{k\neq j}|c_k|^2\right)^{\frac12}=\sigma+O\left(\frac{1}{\sigma^{O(1)}}T^{A_j^+}e^{-\sqrt{3}T}\right)\ .
$$
Moving back to the coordinates $b_1,\ldots,b_N$, this means that we have
\begin{eqnarray*}
b_{\leq j-1}(2\tau_0)&=&O\left(\frac{1}{\sigma^{O(1)}}T^{A_j^+}e^{-2\sqrt{3}T}\right)\\
b_j(2\tau_0)&=&\left[\sigma+\Re O\left(\frac{1}{\sigma^{O(1)}}T^{A_j^+}e^{-\sqrt{3}T}\right)\right]e^{i\vartheta^{(j)}(2\tau_0)}\\
b_{j+1}(2\tau_0)&=&\left[\sqrt{1-\sigma^2}+\Re O\left(\frac{1}{\sigma^{O(1)}}T^{A_j^+}e^{-\sqrt{3}T}\right)\right]\bar\omega e^{i\vartheta^{(j)}(2\tau_0)}+O\left(\frac{1}{\sigma^{O(1)}}T^{2A_j^+}e^{-4\sqrt{3}T}\right)\\
b_{\geq j+2}(2\tau_0)&=&e^{i\vartheta^{(j)}(2\tau_0)}e^{-2\sqrt{3}T}\hat G_{\geq j+2}(2\tau_0)a_{\geq j+2}+O\left(\frac{1}{\sigma^{O(1)}}T^{2A_j^+}e^{-3\sqrt{3}T}\right)\ .
\end{eqnarray*}
where the notation $f=\Re O(\cdot)$ means that both $f=O(\cdot)$ and $f\in\R$. We now have to recast this in terms of the variables $c_1^{(j+1)},\ldots,c_N^{(j+1)}$ in phase with $\T_{j+1}$. Following \cite{CKSTT}, we denote these variables by $\tilde c_1,\ldots,\tilde c_N$. We first note that
$$
\vartheta^{(j+1)}(2\tau_0)=\vartheta^{(j)}(2\tau_0)+\bar\omega+O\left(\frac{1}{\sigma^{O(1)}}T^{2A_j^+}e^{-4\sqrt{3}T}\right)\ .
$$
Then, we deduce our final estimates
\begin{eqnarray*}
\tilde c_{\leq j-1}(2\tau_0)&=&O\left(\frac{1}{\sigma^{O(1)}}T^{A_j^+}e^{-2\sqrt{3}T}\right)\\
\tilde c_j^-(2\tau_0)&=&\sigma+O\left(\frac{1}{\sigma^{O(1)}}T^{A_j^+}e^{-\sqrt{3}T}\right)\\
\tilde c_j^+(2\tau_0)&=&O\left(\frac{1}{\sigma^{O(1)}}T^{2A_j^+}e^{-4\sqrt{3}T}\right)\\
\tilde c_{\geq j+2}(2\tau_0)&=&\omega e^{-2\sqrt{3}T}\hat G_{\geq j+2}(2\tau_0)a_{\geq j+2}+O\left(\frac{1}{\sigma^{O(1)}}T^{2A_j^+}e^{-3\sqrt{3}T}\right)\ .
\end{eqnarray*}
This, together with \eqref{gigi}, shows that the outgoing target $(M_j^+,d_j^+,r_j^+)$ covers the next incoming target $(M_{j+1}^-,d_{j+1}^-,r_{j+1}^-)$ (it is enough to choose $a_{\geq j+2}$ appropriately).

\section{Proof of Lemma \ref{lemma.approx1}} 
\label{lemmasublime}

\begin{proof} 
First note that since $a(0) = g(0)$ is assumed to be compactly supported, the solution $a(t)$ to \eqref{NLS} exists globally in time, is smooth with respect to time, and is in $\ell^1(\Z^2)$ in space. 
Write
$$
F(t):= -i \int_0^t \mathcal E(\tau) d\tau, \quad \text{and}\quad d(t):= g(t) +F(t).
$$
Observe that
$$
- i \dot d = \mathcal N(d-F,d-F,d-F,d-F,d-F),
$$
Observe that $g= O_{\ell^1}(B^{-1})$ and $F=O_{\ell^1}(B^{-1-\sigma})$ by hypothesis. In particular we have $d=O_{\ell^1}(B^{-1})$. By multilinearity  and \eqref{stimamultilin} we thus have 
\begin{equation}\label{estasi}
-i \dot d = \mathcal N(d,d,d,d,d)+O_{\ell^1}(B^{-5-\sigma}).
\end{equation} 
Now write $e:= a-d$, recall that $a$ is the solution of the NLS. Then we have  
\begin{equation}\label{paradisiaca}
-i (\dot d+\dot e)= \mathcal N(d+e,d+e,d+e,d+e,d+e)\ .
\end{equation}
Subtracting \eqref{paradisiaca} from \eqref{estasi} (and using \eqref{stimamultilin}) we get
$$
i\dot e=O_{\ell^1}(B^{-5-\sigma})+O_{\ell^1}(B^{-4}\|e\|_1)+O_{\ell^1}(\|e\|_1^5),
$$
so taking the $\ell^1$ norm and differentiating in time we have:
$$
\frac{d}{dt} \| e\|_1\lesssim B^{-5-\sigma} +B^{-4}\|e\|_1+\|e\|_1^5.
$$
We make the bootstrap assumption  that $\| e\|_1 = O(B^{-1})$ for all $t \in  [0, T ]$, so that one can absorb the third term on the right hand side in the second. Gronwall's inequality then gives:
$$
\|e\|_1\leq B^{-1-\sigma}\mathrm{exp}(C B^{-4} t)
$$
for all $t\in [0,T]$. Since $T\ll B^{4}\log B $ we have $ \|e\|_1\ll B^{-1-\sigma/2}$  and the result follows by the bootstrap argument.
\end{proof}

The result of Lemma \ref{lemma.approx1} is that $g(t)$ is a good approximation of a solution to \eqref{NLS} on a time interval of approximate length $B^4\log B$, a factor $\log B$ larger than the interval $[0,B^4]$ for which the solution is controlled by a straightforward local-in-time argument. We choose the exponent $\sigma/2$ for concreteness, but it could be replaced by any exponent between 0 and $\sigma$.

\section{Two-generation sets without full energy transmission}\label{disegni}
We  describe the dynamics associated to the sets $\mathfrak  S^{(2)}$, $\mathfrak S^{(3)}$ given in the introduction. 

In  $\mathfrak  S^{(2)}$ we have  six complex variables $\beta_k$, $k\in \mathfrak S^{(2)}$ and correspondingly six constants of motion, so that the system is integrable. Passing to symplectic polar coordinates $\beta_k= \sqrt{J_k}e^{i \theta_k}$ we find that
 $J_{k_1}-J_{k_2}$, $J_{k_1}-J_{k_3}$  , $J_{k_4}-J_{k_5}$, $J_{k_4}-J_{k_6}$ are constant in time. Then one can study the dynamics reduced to the invariant subspace where all these constants are zero. We are left with four degrees of freedom denoted by $I_1,I_2,\theta_1,\theta_2$, and the Hamiltonian:
 $$
H= 31 (I_1+I_2)^3 -66 I_1I_2(I_1+I_2) + 24I_1^{3/2}I_2^{3/2}\cos(3(\theta_1-\theta_2)) 
 $$
 Then we reduce to the subspace\footnote{this subspace is invariant due to the conservation of mass} where $I_1+I_2=1$ and get the phase portrait of Figure \ref{vaffa}.
  \begin{figure}[!ht]
 \centering
 \begin{minipage}[b]{11cm}
 \centering
 {\psfrag{I}{$I_1$}
 \psfrag{a}[c]{$-\frac23\pi$}
 \psfrag{b}[c]{$\frac23\pi$}
 \psfrag{f}[r]{$\theta_1-\theta_2$}
 \includegraphics[width=11cm]{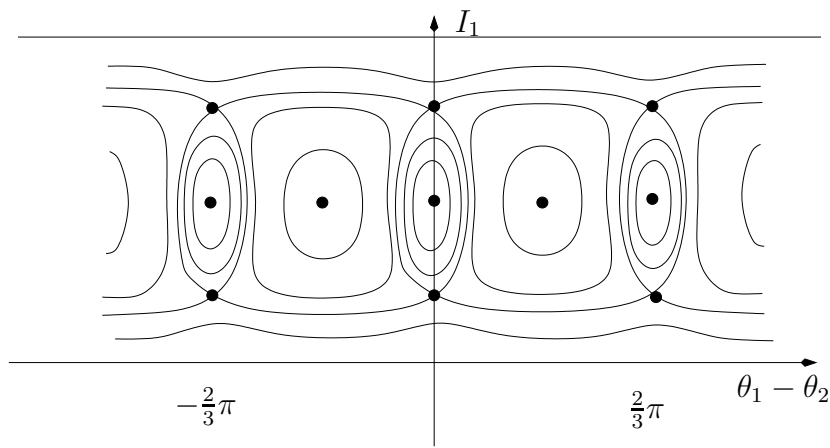}
 }
 \caption{\footnotesize{The phase portrait of $H$ on the subspace $I_1+I_2=1$, the dynamical variables are clearly $I_1,\theta_1-\theta_2$.}}\label{vaffa}
 \end{minipage}
 \end{figure}
  It is evident from the picture that there is no orbit connecting $I_1=0$ to $I_1=1$.  One could argue that this is due to our choice of invariant subspace. However, if we set for instance $J_{k_1}\neq J_{k_2}$ then we cannot transfer all the mass to $k_4,k_5,k_6$ since this would imply $J_{k_1}= J_{k_2}=J_{k_3}=0$.
  
  The case of $\mathfrak S^{(3)}$ is discussed in detail in \cite{GreTho}. Proceeding as above one gets the phase portrait of Figure \ref{vaffa2}.
  \begin{figure}[!ht]
   \centering
   \begin{minipage}[b]{11cm}
   \centering
   {\psfrag{I}{$I_1$}
   \psfrag{a}[c]{$-\frac\pi 2$}
   \psfrag{b}[c]{$\frac\pi 2$}
   \psfrag{f}[r]{$\theta_1-\theta_2$}
   \includegraphics[width=11cm]{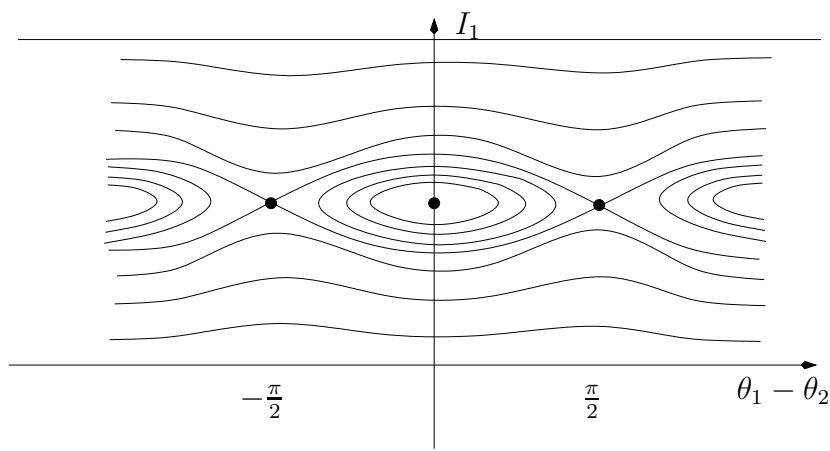}
   }
   \caption{\footnotesize{The phase portrait of $\mathfrak S^{(3)}$.}}\label{vaffa2}
   \end{minipage}
   \end{figure}
 One could generalize this approach by taking two complete and action preserving sets $\SS_1,\SS_2$ and connecting them with resonances as $\mathfrak S^{(2)}$ or $\mathfrak S^{(3)}$ as we have discussed in introduction for $\mathfrak S^{(1)}$.
However the dynamics is in fact qualitatively the same and one does not have full energy transfer. 

We have experimented also with higher order NLS equations. We have not performed a complete classification but it appears that the sets  $\mathfrak S^{(2)}$,  $\mathfrak S^{(3)}$ never give full energy transfer while $\mathfrak S^{(1)}$ does.

\end{appendix}

This research was supported by the European Research Council under
FP7.
\end{document}